\documentclass[12pt,usenames,dvipsnames]{amsart}

\usepackage{comment}
\usepackage{nicematrix}
\DeclareFontFamily{U}{matha}{\hyphenchar\font45}
\DeclareFontShape{U}{matha}{m}{n}{
      <5> <6> <7> <8> <9> <10> gen * matha
      <10.95> matha10 <12> <14.4> <17.28> <20.74> <24.88> matha12
      }{}
\DeclareSymbolFont{matha}{U}{matha}{m}{n}
\DeclareMathSymbol{\odiv}         {2}{matha}{"63}
\usepackage[margin=1in]{geometry}

\usepackage{amsmath,amsfonts,yhmath,hyperref}
\usepackage{amsthm}
\usepackage{amssymb}
\usepackage[latin9]{inputenc}
\usepackage{mathtools}
\usepackage{enumitem}
\usepackage{tikz-cd}
\usepackage{multirow}
\usepackage{makecell}
\usepackage{xcolor}
\usepackage{float}
\usepackage{array}
\usepackage{mathrsfs}

\usepackage{pgf,tikz}

\usetikzlibrary{arrows}
\usetikzlibrary[patterns]
\usetikzlibrary{shapes}
\definecolor{qqqqff}{rgb}{0.,0.,1.}
\definecolor{cqcqcq}{rgb}{0.7529411764705882,0.7529411764705882,0.7529411764705882}
\definecolor{ttqqqq}{rgb}{0.2,0.,0.}
\definecolor{qqqqff}{rgb}{0.,0.,1.}
\definecolor{xdxdff}{rgb}{0.49019607843137253,0.49019607843137253,1.}
\definecolor{zzttqq}{rgb}{0.6,0.2,0.}
\definecolor{cqcqcq}{rgb}{0.7529411764705882,0.7529411764705882,0.7529411764705882}

\definecolor{yqyqyq}{rgb}{0.5019607843137255,0.5019607843137255,0.5019607843137255}
\definecolor{uuuuuu}{rgb}{0.26666666666666666,0.26666666666666666,0.26666666666666666}
\definecolor{xdxdff}{rgb}{0.49019607843137253,0.49019607843137253,1.}
\definecolor{qqqqff}{rgb}{0.,0.,1.}

\definecolor{darkgreen}{rgb}{0.0, 0.5, 0.13}

\newcommand{\PP}{\mathbb{P}}

\newcommand{\ZZ}{\mathbb{Z}}
\newcommand{\RR}{\mathbb{R}}
\newcommand{\CC}{\mathbb{C}}

\newcommand{\TT}{\mathbb{T}}

\newcommand{\B}{\mathcal{B}}

\newcommand{\C}{\mathcal{C}}

\newtheorem{manualtheoreminner}{Theorem}
\newenvironment{manualtheorem}[1]{%
	\renewcommand\themanualtheoreminner{#1}%
	\manualtheoreminner
}{\endmanualtheoreminner}

\newcommand{\Dr}{\operatorname{Dr}}
\newcommand{\Fl}{\operatorname{Fl}}
\newcommand{\FlDr}{\operatorname{FlDr}}

\newcommand{\QDr}{\operatorname{QDr}}
\newcommand{\QGr}{\operatorname{QGr}}

\DeclareMathOperator{\id}{id}
\DeclareMathOperator{\Gr}{Gr}

\DeclareMathOperator{\sign}{sign}
\DeclareMathOperator{\pr}{pr}
\DeclareMathOperator{\supp}{supp}

\newcommand{\trop}{\operatorname{trop}}

\newcommand{\tropbar}{\overline{\operatorname{trop}}}
\newcommand{\defemph}[1]{\emph{#1}}

\newcommand{\val}{\mathrm{val}}

\newtheorem{theo}{Theorem}[section]
\newtheorem{theorem}{Theorem}[section]

\newtheorem{proposition}[theo]{Proposition}

\newtheorem{corollary}[theo]{Corollary}

\newtheorem{lemma}[theo]{Lemma}

\theoremstyle{definition}

\newtheorem{definition}[theo]{Definition}
\newtheorem{proposition-definition}[theo]{Proposition-Definition}

\theoremstyle{remark}
\newtheorem{remark}[theo]{Remark}

\theoremstyle{remark}
\newtheorem{example}[theo]{Example}

\newcommand{\xqed}[1]{
    \leavevmode\unskip\penalty9999 \hbox{}\nobreak\hfill
    \quad\hbox{\ensuremath{#1}}}


\title{Tropical Quiver Grassmannians}
\author{Giulia Iezzi, Victoria Schleis}
\date{}

\begin{document}

\maketitle
\begin{abstract}
    We introduce quivers of valuated matroids  and study their tropical parameter spaces. We define \emph{quiver Dressians}, which parametrize containment of tropical linear spaces after tropical matrix multiplication, and show that tropicalizations of quiver Grassmannians
    parametrize the realizable analogue. We further introduce affine morphisms of valuated matroids and show compatibility with weakly monomial quiver representations. Finally, we show that starting in ambient dimension $2$, quiver Dressians can have nonrealizable points.
\end{abstract}
\section{Introduction}\label{sec intro}

	The Grassmannian $G(r;n)$ parametrizes $r$-dimensional linear subspaces $U$ of an $n$-dimen-sional $K$-vector space $V$, and can be embedded in the projective space $\mathbb{P}^{\binom{n}{r}-1}$ with equations given by the Grassmann-Pl\"{u}cker relations.   The tropicalization $\tropbar(G(r;n))$ of the Grassmannian (as a tropical subvariety of the \emph{tropical projective space} $\mathbb{P}\left(\mathbb{T}^{\binom{n}{r}}\right)$, see Section \ref{sec: prelims tropicalprojectivespace}) parametrizes \defemph{realizable} valuated matroids of rank $r$ on $n$ elements, or equivalently \defemph{realizable} tropical linear spaces (i.e. tropicalizations of linear spaces) of dimension $r$ in $\mathbb{P}(\mathbb{T}^n)$. The object parametrizing all valuated matroids of rank $r$ on $n$ elements, or equivalently all tropical linear spaces of dimension $r$ inside $\mathbb{P}(\mathbb{T}^n)$, is a tropical prevariety $\Dr(r,n)$ called the \emph{Dressian} (see \cite{speyer2008tropicallinearspaces}). Its equations are the tropical Grassmann-Pl\"{u}cker relations, and $\tropbar(G(r;n)) \subseteq \Dr(r,n)$. This  inclusion is strict for large enough $n$. 
 
    Tropical linear spaces correspond to valuated matroids. They are foundational objects which are not only of intrinsic importance for the area of tropical geometry \cite{speyersturmfels2004grassmannian,speyer2008tropicallinearspaces}, in which they appear as the building blocks for tropical manifolds and tropical ideals \cite{maclagan2018tropicalideals} and parametrize hyperplane arrangements, but also connect tropical geometry and matroid theory. Further, they have applications in computational biology as the parameter spaces of phylogenetic trees, \cite{develin2004tropical,maclagan2015book}, and in mathematical physics, describing different combinatorial types of scattering amplitudes \cite{scattering2016,scattering2021}.

    Arrangements of linear spaces that satisfy fixed linear relations to each other can be represented by quivers and parametrized by the quiver Grassmannian. A quiver is a finite directed graph $Q$, with vertices $V$ and arrows $A$ that organizes subspaces $U$ of vector spaces. Each arrow represents a linear map and each vertex a vector space of fixed dimension. Such an assignment of vector spaces is called a \emph{quiver subrepresentation} if $\alpha(U_i)\subseteq U_j$ for all $\alpha$ with source $U_i$ and target $U_j$.
    Quiver Grassmannians are projective varieties parametrizing such subrepresentations. They first appeared in \cite{crawle-boevey1989quiver, schofield1992quiver} and have since been extensively studied. 
    They have been employed in cluster algebra theory \cite{caldero2006cluster} as well as for studying linear degenerations of the flag variety \cite{fourier2020lineardegenerations,cerulliirelli2012quiveranddegenerate,feigin2010grassmanndegenerations, feigin2013frobeniussplittin}.
Notably, every projective variety is isomorphic to a quiver Grassmannian \cite{reineke2013projectiveisquiver,ringel2018quiver}, and every quiver Grassmannian can be embedded in a product of Grassmannians. The equations of this embedding were described in \cite{lorscheid2019pluckerelations}.

An important example of quiver Grassmannians is the flag variety $\Fl(\mathbf{d};n)$, parametrizing flags of linear spaces, i.e. collections of linear subspaces $(U_1,\dots,U_k)$ of $V$ with $\dim(U_i) = d_i$ and $U_1\subseteq U_2 \subseteq \dots \subseteq U_k$, see Example \ref{ex: flag variety}. 

The tropical analogue of flag varieties, the \emph{flag Dressian} $\FlDr(\mathbf{d};n)$, was defined by Haque \cite{haque2012tropical} and further analyzed in \cite{brandt2021tropicalflag}, showing that $\FlDr(\mathbf{d};n)$ parametrizes valuated flag matroids or, equivalently, flags of tropical linear spaces (see  \cite[Theorem A]{brandt2021tropicalflag} and \cite[Theorem 1]{haque2012tropical}).
	
Further, linear degenerations of flag varieties, realized as quiver Grassmannians by using the $A$-type quiver with different linear maps, have already been related to tropical geometry. In \cite{borzi2023lineardegenerate}, Borz\`i and the second author established a similar correspondence to the case of flag matroids for linear degenerate valuated flag matroids and their associated linear degenerate flags of tropical linear spaces.

	In this paper, for a finite quiver $Q$, with a $Q$-representation  $R$ and a fixed dimension vector $\mathbf{d}=(d_1,\dots,d_k)$ we define the \emph{quiver Dressian}  $\QDr(R,\mathbf{d};n)$   as the tropical prevariety 
	cut out by the quiver Pl\"{u}cker relations, see Definition \ref{def: quiver pluecker relations}.
 We prove the following main result, generalizing \cite[Theorem A]{borzi2023lineardegenerate}. 
	
	\begin{manualtheorem}{A}[Theorem \ref{prop: equivalence quiver dressian containment after tropical flag}]\label{thm: main_quiver_correspondence}
		Let $\boldsymbol{\mu}=(\mu_1,\dots,\mu_k)$ be valuated matroids and $Q$ be a finite quiver. Let each arrow $\alpha\in A$ of $Q$ be represented by a matrix $A^{\alpha}$, and let $s(\alpha)$ denote its source and $t(\alpha)$ its target vertex.  The following statements are equivalent: 
		\begin{enumerate}[label=(\alph*)]
			\item $\boldsymbol{\mu} \in \QDr(R,\mathbf{d};n)$;
			\item $\val(A^{\alpha})\odot\tropbar(\mu_{s(\alpha)})\subseteq \tropbar(\mu_{t(\alpha)})$ for all $\alpha \in A$.
		\end{enumerate}
	\end{manualtheorem}
    In words, the quiver Dressian parametrizes tropical linear spaces satisfying the containment conditions described by the quiver.
    For quivers with maps defined by \emph{weakly monomial} matrices (c.f. Definition \ref{def: weaklymonomial}) we connect quiver Dressians to morphisms of valuated matroids:
    \begin{manualtheorem}{B}\label{thm: main_monomial_matrices}
		Let $\boldsymbol{\mu}=(\mu_1,\dots,\mu_k)$ be valuated matroids and $Q$ be a finite quiver whose maps are given as weakly monomial matrices, $\boldsymbol{\mu} \in \QDr(R,\mathbf{d};n)$ is additionally equivalent to
        \begin{enumerate}
			\item[(c)] Each arrow $\alpha\in A$ with $s(\alpha) = U_i$ and $t(\alpha) = U_j$ has an associated contravariant \emph{affine} morphism of valuated matroids $\phi_{\alpha}$  (c.f. Definition \ref{def: affinemorphism})  with $s(\phi_{\alpha}) = \mu_j$ and $t(\phi_{\alpha}) = \mu_i$.
		\end{enumerate}
    \end{manualtheorem}

    We obtain an analogous theorem for the tropicalization of the quiver Grassmannian and realizable tropical linear spaces. 

	\begin{manualtheorem}{C}\label{thm: main_quiver_correspondence_realizable}
		Let $\boldsymbol{\mu}=(\mu_1,\dots,\mu_k)$ be valuated matroids and $Q$ be a finite quiver. Let each arrow $\alpha\in A$ of $Q$ be represented by a matrix $A^{\alpha}$, and let $s(\alpha)$ denote its source and $t(\alpha)$ its target vertex. Let all realizations be in algebraically closed fields with nontrivial valuation.  The following statements are equivalent:
		\begin{enumerate}[label=(\alph*)]
			\item $\boldsymbol{\mu} \in \tropbar(\QGr(R,\mathbf{d};n))$;
			\item $\val(A^{\alpha})\odot\tropbar(\mu_{s(\alpha)})\subseteq \tropbar(\mu_{t(\alpha)})$ for all $\alpha \in A_Q$ and there is a quiver subrepresentation $(N_i)_{i\in V}$ of $Q$ such that $\tropbar(\mu_i) = \tropbar(N_i)$.
        \end{enumerate}
        Further, if all realizations $A^{\alpha}$ are weakly monomial matrices, the above are equivalent to:
        \begin{enumerate}
            \item[(c)] Each arrow $\alpha\in A$ with $s(\alpha) = U_i$ and $t(\alpha) = U_j$ has an associated contravariant  \emph{realizable affine} morphism of valuated matroids $\phi_{\alpha}$  (c.f. Definition \ref{def: affinemorphism})  with $s(\phi_{\alpha}) = \mu_j$ and $t(\phi_{\alpha}) = \mu_i$.
		\end{enumerate}
	\end{manualtheorem}
    We analyze the realizability of points in quiver Dressians, i.e. the relation between tropicalized quiver Grassmannians and quiver Dressians. 

    \begin{manualtheorem}{D}\label{thm: main_realizability}
    For $n=1$, for any finite quiver $Q$ and any $Q$-representation $R$, $\QDr(R,\mathbf{d};n)=\tropbar(\QGr(R,\mathbf{d};n))$.
    
    For any $n\geq 2$, $\tropbar(\QGr(R,\mathbf{d};n))\subseteq\QDr(R,\mathbf{d};n)$ and there exists a quiver $Q$ and a $Q$-representation $R$ where the containment is strict.
    \end{manualtheorem}
    Maps of tropical linear spaces have been studied in \cite{crowley2020modules,frenk2013thesis,mundinger2018image}.

    Flag matroids and their correspondences were generalized to matroids over tracts in \cite{jarra2022flag}. In the upcoming work \emph{Quiver Matroids} \cite{jarra2022qivermatroids}, Jarra, Lorscheid, and Vital study  morphisms of matroids over the more general class of \emph{matroids over perfect idylls}, and construct quiver matroids in this setting. 
 
	This paper is structured as follows. In Section \ref{sec prelims tropgeo} and \ref{sec prelims valuated} we fix our notation and recall the definitions of tropical linear spaces. In Section \ref{sec maps}, we study different notions of linear maps in tropical geometry and introduce affine morphisms of valuated matroids. In Section \ref{sec quiver grassmannian}, we recall classical quiver representations. In Sections \ref{sec tropicalized quiver grassmannian} and \ref{subsec quiver dressians} we arrange our matroids using quivers, define the quiver Dressian and prove the structure results Theorems \ref{thm: main_quiver_correspondence}, \ref{thm: main_monomial_matrices}, and \ref{thm: main_quiver_correspondence_realizable}. Finally, in Section \ref{sec realizability}  we give quiver representations whose quiver Dressians do not coincide with their tropicalized quiver Grassmannian, Examples \ref{fig: nonrealizable n 2}, \ref{ex nonrealizable n = 3} and \ref{ex:nonrealizable quiver n = 4} to prove the statements on realizability of points, Theorem \ref{thm: main_realizability}.

\begin{figure}
    \centering
 \begin{tabular}{c c}\begin{minipage}{6cm}
 \begin{tikzcd}[ampersand replacement=\&, column sep= large, row sep = tiny]
 \& \underset{\textcolor{red}{[2]}}{\overset{\CC^4}{\bullet}} \ar[dr,"\id"] \& \\
\underset{\textcolor{darkgreen}{[1]}}{\overset{\CC^4}{\bullet}} \ar[ur,"\id"] \ar[dr, "\id"]\& \& \underset{\textcolor{blue}{[3]}}{\overset{\CC^4}{\bullet}} \\
\& \underset{\textcolor{orange}{[2]}}{\overset{\CC^4}{\bullet}} \ar[ur,"\id"]
\end{tikzcd} \end{minipage}   & \begin{minipage}{6cm}
\includegraphics[]{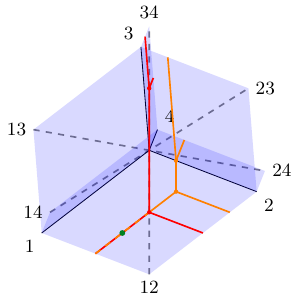}
\end{minipage} 	
 \end{tabular}
    \caption{The numbers below the vertices of the quiver represent the fixed dimensions of the corresponding subspaces. Any subrepresentation consisting of such subspaces describes the containment of a point in two lines, which are both contained in a common plane. On the right, a collection of tropical linear spaces satisfying these conditions.}
    \label{fig:enter-label}
\end{figure}
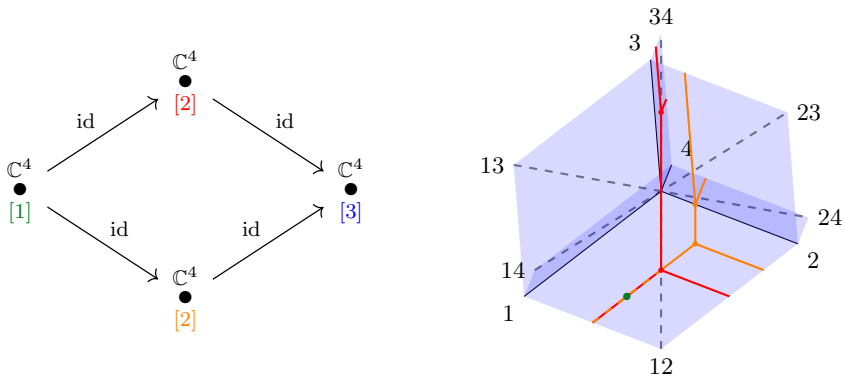

\begin{center}\textsc{Acknowledgements}
\end{center}

We thank Xin Fang, Ghislain Fourier, Martina Lanini, Oliver Lorscheid, Manoel Jarra, Hannah Markwig and Eduardo Vital for helpful discussions. 
The second author further thanks Alessio Borz\`i, Alex Fink and Felipe Rinc\'on for helpful pointers and discussions. Part of the work on this project was done while the first author visited the University of T\"{u}bingen in the Spring of 2023. Both authors were supported by the Deutsche Forschungsgemeinschaft (DFG, German Research
	Foundation), Project-ID 286237555, TRR 195.
\section{Tropical linear spaces and their linear maps}\label{sec: prelims tropicalprojectivespace}
	We assume the reader to be familiar with basic matroid theory, see \cite{oxley2011, welsh1976mtheory}. We fix the notation $[n] = \{1,2,\dots, n\}$ and $\binom{[n]}{r} = \{S \subset [n] : |S|=r\}$. We write $\mathbf{1}=(1,\dots,1)\in\mathbb{R}^{n}$. Every valuation is assumed to be non-Archimedean.
 
    \subsection{Tropical Geometry}\label{sec prelims tropgeo}
    
	In this section, we review the basics of tropical geometry, see \cite{maclagan2015book}, following the min-convention.
	
	We define the tropical numbers as $\mathbb{T} = \mathbb{R} \cup \{ \infty \}$ and consider the semifield $(\mathbb{T},\oplus,\odot)$, called the \defemph{tropical semifield}, where the operations are $a \oplus b = \min\{a,b\}$ and $a \odot b = a+b$ for every $a,b \in \mathbb{T}$. For multiplication with the tropical multiplicative inverse, we write $a\odiv b = a-b$. The \defemph{tropical projective space} is $\mathbb{P}(\mathbb{T}^{n}) = (\mathbb{T}^{n} \setminus \{ (\infty,\dots,\infty) \}) / \mathbb{R} \mathbf{1} = (\mathbb{T}^{n} \setminus \{ (\infty,\dots,\infty) \}) / \sim$. Here, $\sim$ is the equivalence relation $\mathbf{u} \sim \mathbf{v}$ if $\mathbf{u} = \mathbf{v} + c \mathbf{1}$ for some $c \in \mathbb{R}$. A \defemph{tropical polynomial} is an element of the semiring $\mathbb{T}[x_{1},\dots,x_n]$ in the variables $x_{1},\dots,x_n$ with coefficients in $\mathbb{T}$. 
	The \defemph{tropical hypersurface} of a tropical polynomial $F = \bigoplus_{u \in \mathbb{N}^{n}} c_u \odot x^u \in \mathbb{T}[x_{1},\dots,x_n]$ is defined as
	\[ V(F) = \Bigg\{ x \in \mathbb{P}(\mathbb{T}^n) : \min_{ u \in \mathbb{N}^{n}}\left\{ c_u + \sum_{i=1}^n u_i \cdot x_i \right\}\text{is achieved at least twice} \Bigg\} \]
	where, whenever $\min_{u\in \mathbb{N}^{n}}\{ c_u + \sum_{i=1}^n u_i \cdot x_i \} = \infty$, we adopt the convention that the minimum is achieved  at least twice, even if the expression is a tropical monomial. The \defemph{tropical variety} of an ideal of tropical polynomials $J \subseteq \mathbb{T}[x_{1},\dots,x_n]$ is then defined by
	\[ V(J) = \bigcap_{F \in J} V(F) \subseteq \mathbb{P}(\mathbb{T}^n). \]
	In the following, let $K$ be a field with valuation $\val:K \rightarrow \mathbb{T}$. The \defemph{tropicalization} of a polynomial $f = \sum_{u \in \mathbb{N}^{{n}}} a_u x^u \in K[x_{1},\dots,x_n]$ is the tropical polynomial
	\[ \trop(f) = \bigoplus_{u \in \mathbb{N}^{{n}}} \val(a_u) \odot x^u \in \mathbb{T}[x_{1},\dots,x_n]. \]
	
	The \defemph{tropicalization} $\trop(I)$ of an ideal $I \subseteq K[x_1,\dots,x_n]$ is the ideal of tropical polynomials generated by the tropicalizations of all polynomials in $I$:
	\[ \trop(I) = \{ \trop(f) : f \in I \} \subseteq \mathbb{T}[x_1,\dots,x_n]. \]
	Over an algebraically closed field $K$ with a non-trivial valuation, the \defemph{tropicalization} $\tropbar(X)$ of a subvariety $X \subseteq \mathbb{P}_K^n$, is defined by
	\[ \tropbar(X) = \overline{ \{ (\val(x_1),\dots,\val(x_{n})) \in \mathbb{P}(\mathbb{T}^{n}) : [x_1:\dots:x_{n}] \in X \}}, \]
	where the closure is with respect to the Euclidean topology induced on $\mathbb{P}(\mathbb{T}^{n})$.
	
	It follows that there are two possible ways of constructing a tropical variety from a homogeneous ideal $I \subseteq K[x_1,\dots,x_n]$: we can first tropicalize the ideal, and then take its tropical variety $V(\trop(I))$, or we can consider the projective variety $V(I)$ and tropicalize it to obtain $\tropbar(V(I))$. The Fundamental Theorem of Tropical Geometry assures us that, over algebraically closed fields with nontrivial valuation, these two operations yield the same result, i.e. that $ \tropbar(V(I)) = V(\trop(I))$, see \cite[Theorem 6.2.15]{maclagan2015book}. Note that  $\tropbar(X)$ can contain points in which some of the coordinates are $\infty$. For a thorough description of projective tropical varieties, we refer to \cite[Section 2]{brandt2021tropicalflag} and \cite[Section 6.2]{maclagan2015book}, and to \cite{shaw2013matroidfan} for the construction of tropical projective spaces.

	\subsection{Valuated matroids and tropical linear spaces}\label{sec prelims valuated}
	\begin{definition}\label{def valuated matroid}
		A \defemph{valuated matroid} of rank $r$ on the ground set $[n]$ is a function $\nu: \binom{[n]}{r} \rightarrow \mathbb{T}$ such that $\nu(B) \neq \infty$ for some $B \in \binom{[n]}{r}$ and, for all $I,J \in \binom{[n]}{r}$ and $i \in I\setminus J$, there exists $j \in J \setminus I$ satisfying
		\[ \nu(I)+\nu(J) \geq \nu((I \setminus i) \cup j) + \nu((J \setminus j) \cup i). \]
	\end{definition}
	Given a valuated matroid $\nu: \binom{[n]}{r} \rightarrow \mathbb{T}$, the set $\{ B \in \binom{[n]}{r} : \nu(B) \neq \infty \}$ forms the bases of a matroid $N$, which we call the \defemph{underlying matroid} of $\nu$. We say that two valuated matroids $\mu$ and $\nu$ on a common ground set $[n]$ are \emph{equivalent} if there exists $a \in \mathbb{R}$ such that $\mu(B) = \nu(B) + a$ for every $B \in \binom{[n]}{r}$. In other words, every equivalence class of a valuated matroid $\nu: \binom{[n]}{r} \rightarrow \mathbb{T}$ can be seen as a point in  $\mathbb{P}\left(\mathbb{T}^{\binom{n}{r}}\right)$. Throughout, we only consider valuated matroids up to equivalence.
 
    Every linear space has an associated valuated matroid, constructed as follows. Consider a field $K$
	with valuation $\val:K \rightarrow \mathbb{T}$ and an $r$-dimensional vector subspace $L$ of $K^n$ given as the minimal row span of a matrix $A$. We denote by $(p_I)_I$ the \emph{Pl\"ucker coordinates} of $A$, where $p_I$ is the minor of $A$ indexed by $I\in\binom{[n]}{r}$. Then, the function $\mu(A): \binom{[n]}{r} \rightarrow \mathbb{T}$ defined by $I \mapsto \val(p_I)$ is a valuated matroid and we denote its underlying matroid by $M(A)$. Valuated matroids arising in this way are called \defemph{realizable} (over $K$). 
	
	\begin{definition}[{\cite[Definition 4.2.2, Proposition 7.4.7]{brandt2021tropicalflag,brylawski1986matroidquotients}}]
		Let $\mu$ and $\nu$ be two valuated matroids on the ground set $[n]$ of rank $r \leq s$ respectively. We say that $\mu$ is a \defemph{valuated matroid quotient} of $\nu$, denoted $\mu \twoheadleftarrow \nu$, if for every $I \in \binom{[n]}{r}$, $J \in \binom{[n]}{s}$ and $i \in I \setminus J$, there exists $j \in J \setminus I$ such that
		\[ \mu(I) + \nu(J) \geq \mu(I \cup j \setminus i) + \nu(J \cup i \setminus j). \]
	\end{definition}
 
	Matroid quotients describe containment of linear spaces. If $L_1 \subseteq L_2$ are two linear subspaces of $K^n$ generated as the row span of two matrices $A_1$ and $A_2$ respectively, the induced matroids form a quotient, $\mu(A_1) \twoheadleftarrow \mu(A_2)$ (See \cite[Example 4.1.2]{brandt2021tropicalflag}). Valuated matroid quotients arising in this way are called \defemph{realizable} (over $K$). In general, a matroid quotient of two realizable matroids is not necessarily realizable (see \cite[Section 1.7.5, Example 7]{borovik2003coxetermatroids}).

	\begin{definition}\label{def: valuated circuit}
		Let $\mu$ be a valuated matroid of rank $r$ on $[n]$. For each $I \in \binom{[n]}{r+1}$ define an element $C_\mu(I) \in \mathbb{T}^n$ by
		\[ C_\mu(I)_i = \begin{cases}
			\mu(I \setminus i)  & i \in I, \\
			\infty & i \notin I.
		\end{cases}
		\]
		The set of \defemph{valuated circuits} $\mathcal{C}(\mu)$ of $\mu$ is defined as the image of
		\[ \left\{ C_\mu(I) : I \in \binom{[n]}{r+1} \right\} \setminus \{ (\infty,\dots,\infty) \}. \] 
        in $\mathbb{P}(\mathbb{T}^n)$.
 	\end{definition}
	\begin{definition}\label{def: tropicallinearspace}
		Let $\mu$ be a valuated matroid on $[n]$. The \defemph{tropical linear space} of $\mu$ is the tropical variety
		\[ \tropbar(\mu) = \bigcap_{C \in \mathcal{C}(\mu)} V \left( \bigoplus_{i \in [n]} C_i \odot x_i  \right) \subseteq \mathbb{P}(\mathbb{T}^n). \]
	\end{definition}
	
	\begin{definition}
		Let $r \leq n$ be a nonnegative integer. The \defemph{Grassmann-Pl\"{u}cker relations} are  polynomials in the variables $\{p_I : I \in \binom{[n]}{r} \}$ with coefficients in $K$,
		\[ \mathscr{P}_{r;n} = \left\{ \sum_{j \in J \setminus I} \sign(j;I,J) p_{I \cup j} p_{J \setminus j} : I \in \binom{[n]}{r-1}, J \in \binom{[n]}{r+1} \right\},  \]
		where $\sign(j;I,J) = (-1)^{ |\{ j' \in J : j < j' \}| + |\{ i \in I : i > j \}|  }$. These relations define the image of the Grassmannian $\Gr(r;n)$ in the projective space $\mathbb{P}^{\binom{n}{r}-1}$ via the Pl\"{u}cker embedding. The tropicalizations of the Pl\"{u}cker relations are denoted by $\mathscr{P}^{\trop}_{r;n}$.
	\end{definition}
	
	\begin{remark}\label{rmk:contrasting dressian and grassmannian}
		Let $K$ be an algebraically closed field with nontrivial valuation. Grassmannians over $K$ have two analogues in tropical geometry. \emph{Tropical Grassmannians} $\tropbar(G(r;n))$ are tropicalizations of their classical analogues. These \emph{tropical varieties} parametrize \emph{tropicalized} objects, i.e. tropicalizations of linear subspaces of $K^n$ of dimension $r$ (see \cite[Theorem 3.8]{speyersturmfels2004grassmannian}).
		
		On the other hand, the \emph{Dressians} $\Dr(r;n)$ are the intersections of the tropical hypersurfaces given by the Pl\"ucker relations. They are \emph{tropical prevarieties}, i.e. intersections of tropical hypersurfaces. Dressians parametrize \emph{tropical} objects, i.e. tropical linear spaces in $\mathbb{P}(\mathbb{T}^n)$ as in Definition \ref{def: tropicallinearspace}. 
	\end{remark}
	
    \subsection{Morphisms of valuated matroids, matrix multiplication and tropical linear spaces}\label{sec maps}

	In the following, let $K$ be a field with (potentially trivial) valuation $\val:K \rightarrow \mathbb{T}$.
	
\subsubsection{Images of linear spaces under matrix multiplication}\label{sec matrixmult} After fixing bases for each vector space assigned to a vertex in the quiver representation, each map of the quiver representation can be given as a matrix.
Hence, we first study the behavior of tropical linear spaces under matrix multiplication. For a matrix $A$, we write $\val(A):=\big(\val(A_{ij})\big)_{ij}$ for the matrix with valuation applied to all entries, and write $\val(A)\odot \tropbar(\mu)$ for the pointwise tropical matrix multiplication of $\val(A)$ with $\tropbar(\mu)$.

\begin{lemma}\label{lem: image is tropically convex}
    Let $\tropbar(\mu)$ be a tropical linear space in $\mathbb{P}(\mathbb{T}^n)$ and $A\in K^{n\times m}$. Then $\val(A)\odot \tropbar(\mu)$ is tropically convex.
\end{lemma}
\begin{proof}
   Let $\Tilde{v}, \Tilde{w} \in \val(A) \odot \tropbar(\mu)$ and $\lambda,\rho \in \RR$. We need to show that $\lambda\odot\Tilde{v}\oplus\rho \odot\Tilde{w}\in \val(A) \odot \tropbar(\mu),$ i.e. that $\val(A)\odot \tropbar(\mu)$ is tropically convex as defined in the introduction of \cite{develin2004tropical}. By definition of $\val(A) \odot \tropbar(\mu)$, there exist $v,w$ such that $\Tilde{v}= \val(A)\odot v$ and $\Tilde{w} = \val(A) \odot w$. By \cite[Theorem 1.1]{hampe2015convex}, tropical linear spaces are tropically convex, thus $\lambda\odot v\oplus\rho \odot w\in\tropbar(\mu)$. Since tropical matrix multiplication is distributive and commutes with tropical scalar multiplication, \begin{align*}
(\lambda\odot\Tilde{v})\oplus(\rho \odot\Tilde{w}) &= \big(\lambda\odot \big(\val(A)\odot v\big)\big)\oplus\big(\rho \odot \big(\val(A) \odot w\big)\big)\\ &= \val(A)\odot(\lambda\odot v\oplus\rho \odot w)\in \val(A)\odot\tropbar(\mu).  \qedhere     
   \end{align*}
\end{proof} 

In general, images of tropical linear spaces under pointwise matrix multiplication are not tropical linear spaces, as can be seen in the next example.
\begin{example}\label{ex1}
  We consider the trivially valued matroid  given by the map $\mu:\binom{[3]}{2}\rightarrow\RR\cup\{\infty\}, I \mapsto 0$, and the matrices 
  \[A_1 = \begin{bmatrix}
      1 & 1 & 0 \\
      0 & 1 & 0 \\
      0 & 0 & 1
  \end{bmatrix}, A_2 =\begin{bmatrix}
      1 & 1 & 0 \\
      0 & 1 & 1 \\
      0 & 0 & 1
  \end{bmatrix}. \]
  The tropical linear space $\tropbar(\mu)$ and the polyhedral complexes $\val(A_1)\odot\tropbar(\mu)$ and $\val(A_2)\odot\tropbar(\mu)$ are depicted below. Both  $\val(A_1)\odot\tropbar(\mu)$ and $\val(A_2)\odot\tropbar(\mu)$ are not tropical linear spaces, as these polyhedral complexes cannot be assigned balanced weights.
\begin{center}
  \begin{tabular}{c c c}
      \begin{tikzpicture}[scale = 0.5]
    \draw (0,0) -- (3,0) (0,0) -- (0,3) (0,0) -- (-2,-2);
\end{tikzpicture} &  \begin{tikzpicture}[scale = 0.5]
    \draw  (0,0) -- (0,3) (0,0) -- (-2,-2);
    \draw[opacity = 0] (0,0) -- (3,0) (0,0) -- (0,3) (0,0) -- (-2,-2);
\end{tikzpicture} & \begin{tikzpicture}[scale = 0.5]
    \draw   (0,0) -- (-2,-2);
    \draw[opacity = 0] (0,0) -- (3,0) (0,0) -- (0,3) (0,0) -- (-2,-2);
\end{tikzpicture}
    \\
    $\tropbar(\mu)$ &$\val(A_1)\odot\tropbar(\mu)$& $\val(A_2)\odot\tropbar(\mu)$
  \end{tabular}
    
\end{center}
\end{example}

\subsubsection{Affine morphisms of valuated matroids}\label{sec affine morphisms} In the previous section, we described tropical linear maps via tropical matrix multiplication. As tropical linear spaces can be equivalently given as valuated matroids, we now define \emph{affine morphisms of valuated matroids} as another characterization of linear maps between tropical linear spaces.  To consider projection maps, it is necessary to add an additional distinguished element to the ground set to model the origin of vector spaces. 

\begin{definition}\label{def: pointed}
    Let $\mu$ be a valuated matroid over $[n]$. The \emph{pointed valuated matroid} $\mu_o$ over $[n]\cup \{o\}$ is the valuated matroid $\mu\oplus U_{0,1}$, obtained by adding a loop $o$ to the matroid $\mu$. By a slight abuse of notation, we set $\tropbar(\mu_o) :=\tropbar(\mu_o)|_{[n]} = \tropbar(\mu)$, since the tropical linear spaces only differ by removing the $\infty$-entry in the $o$-coordinate. 
\end{definition}

\begin{proposition-definition}[{\cite{borzi2023lineardegenerate,brandt2021tropicalflag}}]
Let $\mu$ and $\nu$ be valuated matroids over the ground set $[n]$ and $f:[n]\cup \{o\}\rightarrow [n] \cup \{o\}$ be a map of sets satisfying $f(o) = o$. Then, we can define the induced matroid $f^{-1}(N)$ via the rank function 
$$\mathrm{rk}_{f^{-1}(N)}(S) := \mathrm{rk}_{N}\big(f(S)\big) \text{ for all } S\subseteq [n]\cup \{o\}$$
where $N$ is the underlying matroid of $\nu_o$. Further, $\nu$ induces a valuation $f^{-1}(\nu)$ with underlying matroid $f^{-1}(N)$, given as $f^{-1}(\nu)(B) = \nu|_{f([n] \cup \{o\})}(f(B)).$ We say that $f:\mu\rightarrow\nu$ is a \defemph{morphism of valuated matroids}  if  $f^{-1}(\nu) \twoheadleftarrow \mu$ is a quotient of valuated matroids, i.e.  for all  $I \in \B(f^{-1}(N))$, $J \in \B(M)$ and $i \in I \setminus J$, there exists $j \in J \setminus I$ such that
	\begin{align*}
    \nu\big(f(I)\big) + \mu(J) \geq \nu\big(f(I \cup j \setminus i)\big) + \mu(J \cup i \setminus j).
	\end{align*} 
\end{proposition-definition}

We now extend this notion slightly, keeping track of additional scaling factors.

\begin{definition}\label{def affine induced}
    Let $\nu$ be a valuated matroid on the ground set $[n]$ and a map $f$ defined as $$(f_1,f_2):[n] \cup \{o\}\rightarrow[n] \cup \{o\}\times \mathbb{T}.$$ We define the \emph{affine induced valuated matroid} as  $$f^{-1}(\nu)(B) = \nu|_{f_1([n] \cup \{o\})}\big(f_1(B)\big)+\sum_{i\in B}f_2(i),$$ where $ \nu|_{f_1([n] \cup \{o\})}$ denotes the restriction of $\nu_o$ to the set $f_1([n] \cup \{o\})$. The affine induced valuated matroid is a pointed valuated matroid as in Definition \ref{def: pointed}, hence its tropical linear space is defined as $\tropbar\big(f^{-1}(\nu)\big) := \tropbar\big(f^{-1}(\nu)\big)|_{[n]}$.
\end{definition}

\begin{lemma}\label{lemma affine morphism well defined}
    Let $\nu$ be a valuated matroid on the ground set $[n]$ and $f:[n]\cup \{o\}\rightarrow[n]\cup\{o\}\times \mathbb{T}$. Then, the affine induced valuated matroid is a valuated matroid as defined in Definition \ref{def valuated matroid}. 
\end{lemma}
\begin{proof}
    By \cite[Propositon 3.18, Definition 3.19]{borzi2023lineardegenerate}, the map of sets $f_1$ induces a valuated matroid $f_1^{-1}(\nu)(B) = \nu|_{f_1([n] \cup \{o\})}\big(f_1(B)\big)$ with underlying matroid $f_1^{-1}(N)$. Since $f_1^{-1}(\nu)$ is a valuated matroid,  for all $I,J \in \binom{[n]}{r}$ and $i \in I\setminus J$ there exists $j \in J \setminus I$ satisfying
		\begin{align*}&\nu|_{f_1([n] \cup \{o\})}\big(f_1(I)\big)+\nu|_{f_1([n] \cup \{o\})}\big(f_1(J)\big)\\ \geq \hspace{2pt} &\nu|_{f_1([n] \cup \{o\})}\big(f_1(I \setminus i \cup j)\big) +\nu|_{f_1([n] \cup \{o\})}\big(f_1(J \setminus j \cup i)\big).
		\end{align*} 
    Now \begin{align*}
     \sum_{i'\in  I}f_2(i') +\sum_{j\in  J}f_2(j') = \sum_{i'\in ((I\setminus i) \cup j)}f_2(i') +\sum_{j'\in  ((J \setminus j)\cup i)}f_2(j'),  
    \end{align*}
    thus 
    \begin{align*}
        &\nu|_{f_1([n] \cup \{o\})}\big(f_1(I)\big)+ \sum_{i'\in  I}f_2(i') +\nu|_{f_1([n] \cup \{o\})}(f_1(J))  +\sum_{j'\in  J}f_2(j')\\
        \geq  & \nu|_{f_1([n] \cup \{o\})}(f_1((I \setminus i) \cup j)) +\sum_{i'\in ((I\setminus i) \cup j)}f_2(i') +\nu|_{f_1([n] \cup \{o\})}(f_1((J \setminus j) \cup i))+\sum_{j'\in  ((J \setminus j)\cup i)}f_2(j'),
    \end{align*}
    and $f^{-1}(\mu)(B)$ defines a valuation by Definition \ref{def valuated matroid}. 
\qedhere
\end{proof}

For an example of an affine induced valuated matroid, see Example \ref{ex morphism and monomial matrix}.

\begin{definition}\label{def: affinemorphism}
    Let $\mu$ and $\nu$ be valuated matroids over the ground set $[n]$ and $f:[n] \cup \{o\}\rightarrow [n] \cup \{o\}\times \mathbb{T}$ be a map of sets. By abuse of notation, we say that $f:\mu\rightarrow\nu$ is an \defemph{affine} morphism of valuated matroids if  $f^{-1}(\nu) \twoheadleftarrow \mu$ as in Definition \ref{def affine induced} is a quotient of valuated matroids.
    
    If additionally both $\mu$ and $\nu$ are \emph{realizable}, we say that $f$ is a \emph{realizable} affine morphism of valuated matroids if both $f^{-1}(\nu)$  and the quotient  $f^{-1}(\nu) \twoheadleftarrow \mu$ are realizable.
\end{definition}

\subsubsection{Weakly monomial matrices and associated affine maps of matroids}
The two different notions of linear maps between tropical linear spaces we considered in Sections \ref{sec matrixmult} and \ref{sec affine morphisms} are not fully compatible, as tropicalization commutes with monomial maps, but not with linear maps. We now restrict to specific matrices for which matrix multiplication again yields a tropical linear space, and show that these correspond to affine morphisms of valuated matroids. 

\begin{definition}\label{def: weaklymonomial}
    Let $A\in K^{n\times m}$. We call $A$ a \emph{weakly monomial} matrix if $A$ has at most one nonzero entry in each row. 
\end{definition}

\begin{example}\label{ex: multiplying with a tropical linear space}
We consider the trivially valued matroid from Example \ref{ex1}. Its tropical linear space is the classical tropical line with vertex at $(0,0,0)$ in $\RR^3/\RR \mathbf{1}$ depicted in red below. Multiplying by the tropicalization of a weakly monomial matrix $\val(A)$ below yields a shifted tropical linear space to the right, depicted in blue. The permutation further switches the $y$ and $z$ rays. 
\begin{center}
\begin{tabular}{c c}
    \begin{minipage}{5cm}       
\begin{equation*}
    \val(A) = \begin{bmatrix}
        3 & \infty & \infty \\ \infty & \infty & 1 \\ \infty & 0 & \infty
    \end{bmatrix}
\end{equation*}
    \end{minipage}
 &  
    \begin{minipage}{5cm}       
\begin{tikzpicture}[scale = 0.5]
    \draw[thick,red] (0,0) -- (5,0) (0,0) -- (0,3) (0,0) -- (-2,-2);
    \draw[thick,blue] (3,1) --  (5,1) (3,1) -- (3,3) (0,-2) -- (3,1);
    \node[left] at (0,0.4) {$(0,0)$};
    \node[right] at (3,1.6) {$(3,1)$};
\end{tikzpicture}
\end{minipage}

\end{tabular}
\end{center}

\end{example}

To show the compatibility of morphisms of valuated matroids with the multiplication of a tropical linear space by a weakly monomial matrix, we need the following equivalent characterization of tropical linear spaces using vectors generated by cocircuits:
\begin{proposition-definition}[{\cite{maclagan2018tropicalideals,murota2001circuitvaluation}}]\label{thm: characterizations of trop mu}
        Let $\mu$ be a valuated matroid over the ground set $[n]$ of rank $r$. For each $I\in\binom{[n]}{r-1}$, we define $C^*_\mu(I)\in\mathbb{T}^n$ by
        \[C^*_\mu(I)_i = \begin{cases}
            \mu(I\cup i) & i \notin I\\
            \infty & i \in I
        \end{cases}\]
        The \defemph{valuated cocircuits} of $\mu$ are defined as
		\[ \mathcal{C}^*(\mu) = \left\{ C^*_\mu(I) : I \in \binom{[n]}{r-1} \right\} \setminus \{ (\infty,\dots,\infty) \}. \]
    The support of a cocircuit $C^*\in \mathcal{C}^*(\mu)$ is the set $\supp(C^*)=\{i\in[n]: C^*_i \neq \infty \}.$ The tropical span of cocircuits of a valuated matroid  equivalently defines $\tropbar(\mu)$, see \cite[Theorem B]{brandt2021tropicalflag}.  
    \[ \tropbar(\mu) = \left\{ \bigoplus_{C^*\in \mathcal{C}^{*} (\mu)} \lambda_{C^*} \odot C : \lambda_{C^*} \in \mathbb{R} \right\}. \]
\end{proposition-definition}

\begin{lemma}\label{lemma cocircuits induced valuated matroids}
    Equivalently to characterizing bases, we can characterize cocircuits of induced valuated matroids. Let $\mu$ be a valuated matroid on $[n]$ and $f:[n]\cup \{o\}\rightarrow[n]\cup\{o\}\times \mathbb{T}$. Let $I\in\binom{[n]}{\mathrm{rk}(f^{-1}(\mu))-1}$.  Then, the coordinate entries of the valuated cocircuits of $f^{-1}(\mu)$ are given as follows:
            \[C^*_{f^{-1}(\mu)}(I)_i = \begin{cases}
            C^*_{\mu|_{f_1([n]\cup \{o\})}}\big(f_1(I)\big)_{f_1(i)} \odot \bigodot_{k\in I\cup i} f_2(k) &  o \notin f_1(I\cup i),\\
            \infty & o \in f_1(I\cup i).
        \end{cases}\]
\end{lemma}
\begin{proof}
Since $o$ is a loop, every set containing $o$ has valuation $\infty$. Now, assume $o\notin f_1(I\cup i)$. If $i\in I$, then $f_1(i)\in f(I)$ and $C^*_{f^{-1}(\mu)}(I)_i = C^*_{\mu|_{f_1([n]\cup \{o\})}}\big(f_1(I)\big)_{f_1(i)} \odot \bigodot_{k\in I\cup i} f_2(k) = \infty$ by Definition \ref{thm: characterizations of trop mu}. Further, by Definition \ref{thm: characterizations of trop mu}, if $i\notin I$, $C^*_{f^{-1}(\mu)}(I)_i = f^{-1}(\mu)(I\cup i).$ Now, by Definition \ref{def affine induced},
\begin{align*}f^{-1}(\mu)(I\cup i) &= \mu|_{f_1([n] \cup \{o\})}\big(f_1(I\cup i)\big)+\sum_{k\in I\cup i}f_2(k) =\mu|_{f_1([n] \cup \{o\})}\big(f_1(I)\cup f_1(i))\big)+\sum_{k\in I\cup i}f_2(k) \\&=C^*_{\mu|_{f_1([n]\cup \{o\})}}\big(f_1(I)\big)_{f_1(i)} \odot \bigodot_{k\in I\cup i} f_2(k).\qedhere
\end{align*}
\end{proof}

\begin{definition}\label{def: associated map matrix}
       Let $A_f\in K^{n\times n}$ be a weakly monomial matrix.
   We define an \emph{associated map} $f:[n]\cup \{o\} \rightarrow[n]\cup \{o\} \times \mathbb{T}$ by $$i\mapsto\begin{cases}
        (o,\infty) & i = o \text{ or } A_{ij}  = 0 \text{ for all columns } j\\
  \big(j,\val(A_{ij})\big) & A_{ij} \neq 0.
    \end{cases}$$  If $f: [n] \cup \{o\} \rightarrow [n] \cup \{o\}\times \mathbb{T}$ is a map of sets such that for all $i\in [n] \cup \{o\}$, $f_2(i) = \val(k)$ for some $k \in K$, we construct an \emph{associated matrix} $A_f\in K^{n\times n}$ by setting $$A_{ij} = \begin{cases} k & \text{if } f_1(i) = j \text{ and } i,j\neq o\\ 0 &\text{otherwise.}
    \end{cases}$$ 
\end{definition}
\begin{remark}
    In the above definition, the associated map of a weakly monomial matrix is unique. The associated matrix $A_f$ is not unique - but its valuation $\val(A_f)$ is. This is due to the fact that there can be multiple $k\neq k'\in K$ with valuation $\val(k) = \val(k')$.
\end{remark}
\begin{remark}\label{rmk construction associated maps}
    We construct the associated map for some special matrices:
    \begin{itemize}
        \item[(a)] If $A_f$ is a \emph{permutation matrix}, then the associated map $f$ consists of a permutation map $f_1$ that fixes $o$, and $f_2(i)=0$ for $i\in[n]$, and $f_2(o) = \infty$. 
        \item[(b)] If $A_f$ is a \emph{projection matrix} of rank $s<n$, then the associated map $f$ is given as $f(i)=(i,0)$ on $[s]$ and $f(i)=(o,\infty)$ otherwise.
        \item[(c)] If $A_f$ is a \emph{diagonal matrix}, the associated map $f$ is $f(i) = (i,\val(A_{ii}))$ for $i\in[n]$ and $f(o)=(o,\infty)$.
    \end{itemize}
    For the types of maps associated to matrices given above, we can again construct a (potentially not unique) associated matrix, provided that for every $i\in[n]$, there exists a $k\in K$ such that $\val(k)=f_2(i)$.
    \begin{itemize}
        \item[(a')] If $f_1$ is a permutation map fixing $o$ and $f_2(i)=0$ for $i\in[n]$ and $f_2(o)=\infty$, then $A_f$ can be chosen as the permutation matrix associated to the permutation $f_1$.
        \item[(b')] If $\pr_S$ is a projection map satisfying $f(i)=(i,0)$ for $i\in [n]\setminus S$ and $f(i)=(o,\infty)$ for $i\in S\cup o$, a  matrix associated to $\pr_S$ is given as the projection matrix $A_{S_{i,i}}=1$ if $i\notin S$, and $A_{i,j}=0$ otherwise.
        \item[(c')] If $f_1$ is the identity map, an associated matrix $A_f$ is a diagonal matrix with entries $A_{f,ii}=k_i$ for $k_i\in K$ with $\val(k_i)=f_2(i).$
    \end{itemize}
\end{remark}
    We will now develop the correspondence between affine morphisms of valuated matroids and matrix multiplication with weakly monomial matrices. For the proof, we will decompose each weakly monomial matrix (and analogously, each map) into a product of matrices.

    \begin{lemma}\label{lem product of matrices}
        Let $g,h: [n]\cup \{o\}\rightarrow [n]\cup \{o\} \times \mathbb{T}$ be  maps with $g(o) = (o,\infty)$ and $h(o) = (o,\infty)$, and let $A_g,A_h\in K^{n \times n}$ be associated weakly monomial matrices. Assume that for any tropical linear space $\tropbar(\mu)$, $\val(A_g)\odot \tropbar(\mu)= \tropbar(g^{-1}(\mu))$ and $\val(A_h)\odot \tropbar(\mu)= \tropbar(h^{-1}(\mu))$. Then, for $h\circ g (i) = (h_1(g_1(i)), g_2(i)+h_2(g_1(i)))$, we have $\val(A_g\cdot A_h)\odot \tropbar(\mu) = \tropbar((h\circ g)^{-1}(\mu))$.
        
    \end{lemma}
    \begin{proof} By assumption,
         \begin{align*}\val(A_g\cdot A_h)\odot\tropbar(\mu) &= \val(A_g)\odot\val(A_h)\odot\tropbar(\mu)\\&=\val(A_g)\odot(\tropbar(h^{-1}(\mu)))=\tropbar(g^{-1}(h^{-1}(\mu)))
         \end{align*}
        Now, we show that $g^{-1}(h^{-1}(\mu)) = (h\circ g)^{-1}(\mu)$. Let $r=\mathrm{rk}(\mu)$, and  $I\in \binom{[n]}{r}$. Then, 
        \begin{align*}
        (h\circ g)^{-1}(\mu)(I) &= \mu|_{(h\circ g)_1([n] \cup \{o\}))}\big((h\circ g)_1(I)\big)+\sum_{i\in I}(h\circ g)_2(i) \\ &=  \mu|_{(h_1(g_1([n] \cup \{o\})))}\big(h_1(g_1(I)\big)+\sum_{i\in I}\big(g_2(i)+h_2(g_1(i))\big) \\ 
        &= \mu|_{(h_1(g_1([n] \cup \{o\})))}\big(h_1(g_1(I)\big)+\sum_{i\in I}h_2(g_1(i))+\sum_{i\in I}g_2(i)\\
        &= h^{-1}(\mu)\big(\mu|_{g_1([n]\cup\{o\})}(g_1(I))\big)+\sum_{i\in I}g_2(i)=g^{-1}\big(h^{-1}(\mu)\big)(I).        %    g^{-1}(h^{-1}(\mu))(I) &=  \mu|_{f_1([n] \cup \{o\})}\big(f_1(B)\big)+\sum_{i\in B}f_2(i)
        %
        %    \mu|_{g_1(h_1([n] \cup \{o\}))}\big(g_1(h_1(B))\big)+\sum_{i\in }_2(i) = (h\circ g)^{-1}(\mu)(I)
        \end{align*}
        Above, the last and the second to last equalities are obtained by explicitly writing out the maps using Definition \ref{def affine induced}. Since the valuated matroids are equal for each basis, the associated tropical linear spaces coincide.
        We remark that the reversal of order here is compatible with the contravariant properties we claim for affine morphisms of valuated matroids in Theorems \ref{thm: main_monomial_matrices} and \ref{thm: main_quiver_correspondence_realizable}.
    \end{proof}

\begin{lemma}\label{lem mat equals morph 01 }
     Let $\mu$  be a valuated matroid over $[n]$. Let $f: [n]\cup \{o\}\rightarrow [n]\cup \{o\} \times \mathbb{T}$ be a map where $f_2(i)=0$ if $f_1(i)\neq o$. Then, for any associated weakly monomial matrix $A_f$ defined in Definition \ref{def: associated map matrix},  $\tropbar\big(f^{-1}(\mu)\big) = \val(A_f)\odot\tropbar(\mu)$. Conversely, if $A_f\in K^{n\times n}$ is a weakly monomial matrix with entries in $\{0,1\}$, the associated map $f$ satisfies $\tropbar\big(f^{-1}(\mu)\big) = \val(A_f)\odot\tropbar(\mu)$.
\end{lemma}
\begin{proof}
    As a tool to show the general case, we first show 
    \begin{equation}\label{eq induced is matrix mult}
        \tropbar\big(f^{-1}(\mu)\big) = \val(A_f)\odot\tropbar(\mu)
    \end{equation}
    for permutation and projection matrices.
    
    For permutation matrices, the statement is clear, as by Remark \ref{rmk construction associated maps}(a), $f_1$ is a permutation, hence permutes entries of vectors uniformly,  and $f_2(i) = 0$ for $i\in [n]$ and $f_2(o)=\infty$.
    
    For projection matrices, we obtain \eqref{eq induced is matrix mult}  as a corollary of \cite[Proposition 3.14]{borzi2023lineardegenerate} as follows.  For $S\subseteq [n]$ and a projection map $\pr_S:[n]\cup \{ o\}\rightarrow[n]\cup\{o\}$ where $ \pr_S(i) = (i,0)$ if $i\notin S$ and $\pr_S(i)=(o,\infty)$ if $i\in S$, we recover the setting in \cite[Proposition 3.14]{borzi2023lineardegenerate}, by replacing matrix multiplication with the projection matrix $A_S$ from Remark \ref{rmk construction associated maps}(b) with application of the tropical projection map $\pr_S^{\trop}: \mathbb{T}^n \rightarrow \mathbb{T}^n$ defined as
	$ \big( \pr_S^{\trop}(x_1,\dots,x_n) \big)_i =
		x_i  \text{ if } i \notin S$, and $
		\infty \text{ if } i \in S.
	$ Then, by \cite[Proposition 3.14]{borzi2023lineardegenerate}, $A_S\odot\trop(\mu)=\pr_S^{\trop}\big(\tropbar(\mu)\big) = \tropbar(\mu\setminus S \oplus U_{0,|S|}$), i.e. the valuated matroid arising as the deletion of $S$ from $\mu$, substituting all deleted elements with loops. Now, by the construction of $\pr_S$ and Definition \ref{def: affinemorphism}, $\tropbar\big(\pr_S^{-1}(\mu)\big)=\tropbar(\mu\setminus S\oplus U_{0,|S|}),$ thus $A_{\pr_S}\odot\trop(\mu)=\tropbar\big(\pr_S^{-1}(\mu)\big)$.

    Now, let $A_f$ be an arbitrary weakly monomial matrix with entries in $\{0,1\}$. Then, after multiplication with a permutation matrix, $A_f$ is of the form \[A_f=     \left[ 
    \begin{array}{c | c} 
   
     \begin{array}{c c c c}
        \mathbf{1}_{k_1} &\mathbf{0}_{k_1} & \cdots & \mathbf{0}_{k_1} \\ 
        \mathbf{0}_{k_2} & \mathbf{1}_{k_2} & &  \mathbf{0}_{k_2} \\
        \hdots & & \ddots & \hdots \\
        \mathbf{0}_{k_l} & \mathbf{0}_{k_l} &\cdots & \mathbf{1}_{k_l}
    \end{array} & 0
     \end{array} 
    \right], 
  \]
    where $\mathbf{1}_k$ denotes the vector $(1,\dots,1)\in \RR^k$, and by Lemma \ref{lem product of matrices}, it is sufficient to show it for this form. Let $A_{f'}$ be the matrix obtained from $A_f$ by setting all except the first nonzero entry in each column to $0$. Up to permutation, this is a projection matrix, thus $\tropbar\big(f'^{-1}(\mu)\big) = \val(A_{f'})\odot\tropbar(\mu)$.

    Now, let $w \in \val(A_{f})\odot\tropbar(\mu)$. Then, there exists $v\in\tropbar(\mu)$ such that $w = \val(A_{f})\odot v$. Let $w'= \val(A_{f'})\odot v\in\tropbar\big(f'^{-1}(\mu)\big)$. Now, $w_i = (\val(A_f)\odot v)_i = w'_i$ if $A_{f',i}$ is a nonzero row. Otherwise, $w'_i = \infty$ and  $$w_i = (\val(A_{f})\odot v)_i = (\val(A_{f})\odot v)_j = (\val(A_{f'})\odot v)_j = w'_j$$ for a row $j$ with unique nonzero entry in the same column as the row $i$. 
    
    The map $f$ associated to $A_f$ is defined by $f_1(j)=i$ for each $j\in \{i,\dots,i+k_c\}$ where $i$ is the first nonzero entry of $A_f$ for the column $c$, and by $f_2(j)=\infty$ for $j\in\{\sum_{p=1}^{l}k_p,\dots,n\}$. Consequently, by Lemma \ref{lemma cocircuits induced valuated matroids}, each cocircuit $C^*=C^*_{f^{-1}(\mu)}(I)\in\mathcal{C}^*(f_1^{-1}(\mu))$ has coordinates 
    \begin{align}\begin{split}\label{eq cocircuits 01}
        C^*_j&=C^*_{f^{-1}(\mu)}(I)_j=\mu|_{f_1([n] \cup \{o\})}\big(f_1(I)\cup f_1(j)\big)=\mu|_{f_1([n] \cup \{o\})}\big(f_1(I)\cup i\big) \\&= \mu|_{f_1([n] \cup \{o\})}\big(f_1(I)\cup f_1(i)\big) = C^*_i =  C^*_{f'{}^{-1}(\mu)}(I')_i
    \end{split}
    \end{align}
    for all $j\in \{i,\dots,i+k_c\}$, where $I'$ is the set containing the first nonzero entries of each column that each element in $I$ belongs to.  As $\tropbar\big(f^{-1}(\mu)\big)$  and $\tropbar\big(f'^{-1}(\mu)\big)$ are tropically generated by their cocircuits by Proposition-Definition \ref{thm: characterizations of trop mu}, we can write $w'$ as the tropical sum of cocircuits, $w'_i  = \bigoplus_{C^*\in \mathcal{C}^*(f'{}^{-1}(\mu))}\lambda_{C^*}\odot C^*_i$, thus  
    $$w_j = w'_i = \bigoplus_{C^*\in \mathcal{C}^*(f'{}^{-1}(\mu))}\lambda_{C^*}\odot C^*_i \overset{\eqref{eq cocircuits 01}}{=} \bigoplus_{C^*\in \mathcal{C}^*(f{}^{-1}(\mu))}\lambda_{C^*}\odot C^*_j,$$
    where the last equality follows by using Equation \ref{eq cocircuits 01}, and that every cocircuit of $f'{}^{-1}(\mu)$ is already a cocircuit of $f^{-1}(\mu)$. Thus $w\in\tropbar\big(f'^{-1}(\mu)\big).$ The reverse direction follows analogously, thus \eqref{eq induced is matrix mult} holds.
\end{proof}

\begin{lemma}\label{lemma: mat equals morph diags}
      Let $\mu$  be a valuated matroid over $[n]$. Let $f: [n]\cup \{o\}\rightarrow [n]\cup \{o\} \times \mathbb{T}$ be a map satisfying $f_1(i) = i$. Further, assume that for all $i\in [n]$ there exists $k\in K$ such that $f_2(i) = \val(k)$. Then, for the associated weakly monomial matrix $A_f$ defined in Definition \ref{def: associated map matrix},  $\tropbar\big(f^{-1}(\mu)\big) = \val(A_f)\odot\tropbar(\mu)$. Conversely, if $A_f\in K^{n\times n}$ is a full rank diagonal matrix, the associated map $f$ satisfies $\tropbar\big(f^{-1}(\mu)\big) = \val(A_f)\odot\tropbar(\mu)$.
\end{lemma}
\begin{proof}
    Assume $A_f$ is a full-rank diagonal matrix, then the map of sets $f_1$ is the identity. Then, for each valuated cocircuit $C^*_{f^{-1}(\mu)}(J)$ of $f^{-1}(\mu)$, by Lemma \ref{lemma cocircuits induced valuated matroids} there exists a valuated cocircuit $C^*_\mu(J)$ with equal support $J$ such that
    \begin{equation}\label{eq cocircuits}
    C^*_{f^{-1}(\mu)}(J)_i = C^*_{\mu}(J)_i\odot \bigodot_{k\in J \cup i} f_2(k),\end{equation}

    and vice versa. Let $v\in\tropbar\big(f^{-1}(\mu)\big)$. By Proposition \ref{thm: characterizations of trop mu}, $v$ can be written as a tropical linear combination of cocircuits of $f^{-1}(\mu)$, thus

    \begin{align*}
        v_i &=  \bigoplus_{C^*\in\C^{*}(f^{-1}(\mu))}\lambda_{C^*}\odot C_i \overset{\eqref{eq cocircuits}}{=} \bigoplus_{C^*\in\C^{*}(\mu)}\big(\lambda_{C^*}\odot C^*_i\odot \bigodot_{k\in \supp(C^*) \cup i} f_2(k)\big) \\&= f_2(i)\odot\bigg(\bigoplus_{C^*\in\C^{*}(\mu)}\big(\lambda_{C^*}\odot C^*_i\odot\bigodot_{k\in \supp(C^*)} f_2(k)\big)\bigg)\\
        &= \val(A_{f,ii})\odot\bigg(\bigoplus_{C^*\in\C^{*}(\mu)}\big(\large(\lambda_{C^*}\odot\bigodot_{k\in \supp(C^*)} f_2(k)\large)\odot C^*_i\big)\bigg).
    \end{align*}
    We have $\lambda_{C^*}\odot\bigodot_{k\in \supp(C^*)} f_2(k)\in\RR$, thus, by Proposition  \ref{thm: characterizations of trop mu}, the vector $v'$ with entries  $v'_i =\big(\lambda_{C^*}\odot\bigodot_{k\in \supp(C^*)} f_2(k)\big)\odot C^*_i$ is in $\tropbar(\mu)$. As $v=\val(A_f)\odot v'$,  $v\in\val(A_f)\odot\tropbar(\mu)$.  
    
    Now, let $w\in\tropbar(\mu)$. Then
    $w = \bigoplus_{C^*\in\C^{*}(\mu)}\Tilde{\lambda}_{C^*}\odot C^*$ for some fixed $\Tilde{\lambda}_{C^*}\in\RR$ for each cocircuit $C^*$, and we have 
     \begin{align*}
    &(\val(A_f)\odot w)_i = (\val(A_f) \odot \bigoplus_{C^*\in\C^{*}(\mu)}\Tilde{\lambda}_{C^*}\odot C^*)_i \overset{(*)}{=}  \bigoplus_{C^*\in\C^{*}(\mu)} \big(\val(A_{f,ii}) \odot\Tilde{\lambda}_{C^*}\odot C^*_i\big) \\  &\hspace{-2pt}\overset{(**)}{=} \bigoplus_{C^*\in\C^{*}(\mu)} \big(\Tilde{\lambda}_{C^*}\odot f_2(i) \odot C^*_i\big) \\ &\hspace{1pt}= \bigoplus_{C^*\in\C^{*}(\mu)}\bigg(\big(\Tilde{\lambda}_{C^*}\odiv\bigodot_{k\in \supp(C^*)}f_2(k)\big)\odot C^*_i\odot f_2(i)\bigodot_{k\in \supp(C^*)} f_2(k)\bigg) \\
     &\hspace{1pt}\overset{\eqref{eq cocircuits}}{=} \bigoplus_{C^*\in\C^{*}(f^{-1}(\mu))}\bigg(\big(\Tilde{\lambda}_{C^*}\odiv\bigodot_{k\in \supp(C^*)}f_2(k)\big)\odot C^*_i\bigg) 
     \end{align*}
     where $(*)$ follows by $A_f$ being diagonal and $(**)$ follows by the construction of $f_2$. Thus, $\val(A_f)\odot w$ can be written as the tropical sum of scalar multiples of cocircuits of $\tropbar\big(f^{-1}(\mu)\big)$, hence $\val(A_f)\odot w\in\tropbar\big(f^{-1}(\mu)\big).$ 
\end{proof}

\begin{proposition}\label{prop affine morphism is matrix multiplication}
    Let $\mu$  be a valuated matroid over $[n]$. Let $f: [n]\cup \{o\}\rightarrow [n]\cup \{o\} \times \mathbb{T}$ be a map such that for all $i\in [n]$ there exists $k\in K$ satisfying $f_2(i) = \val(k)$. Then, for the associated weakly monomial matrix $A_f$ defined in Definition \ref{def: associated map matrix},  $\tropbar\big(f^{-1}(\mu)\big) = \val(A_f)\odot\tropbar(\mu)$. Conversely, if $A_f\in K^{n\times n}$ is a weakly monomial matrix, the associated map $f$ satisfies $\tropbar\big(f^{-1}(\mu)\big) = \val(A_f)\odot\tropbar(\mu)$.
\end{proposition}
\begin{proof}
     By Lemma \ref{lem mat equals morph 01 }, the statement holds for a square weakly monomial matrix with entries in $\{0,1\}$, and by Lemma \ref{lemma: mat equals morph diags} it holds for diagonal matrices of full rank. Since every square weakly monomial matrix can be written as the product of these two types, by Lemma \ref{lem product of matrices}, the claim follows for all matrices. 
\end{proof}
     \begin{corollary}
    Let $\mu$ be a valuated matroid over $[n]$ and let $A\in K^{n\times n}$ be a weakly monomial matrix. Then, $\val(A)\odot \tropbar(\mu)\subseteq \mathbb{P}(\mathbb{T}^n)$ is a tropical linear space.  
     \end{corollary}

\begin{example}\label{ex morphism and monomial matrix}
In Example \ref{ex: multiplying with a tropical linear space}, we saw that matrix multiplication with a weakly monomial matrix induced a permutation and translation of the tropical linear space. Its matrix can be decomposed into a permutation matrix and a full rank diagonal matrix,

\begin{equation*}
    \val(A) = \begin{bmatrix}
        3 & \infty & \infty \\ \infty & \infty & 1 \\ \infty & 0 & \infty
    \end{bmatrix} = \begin{bmatrix}
        3 & \infty & \infty \\ \infty & 1 & \infty \\ \infty & \infty & 0
    \end{bmatrix} \odot \begin{bmatrix}
        0 & \infty & \infty \\ \infty & \infty & 0 \\ \infty & 0 & \infty
    \end{bmatrix}.
    \end{equation*}

The associated map of matroids is  
 $$ f:[n]\rightarrow[n]\times\RR;
     1\mapsto (1,3); 2\mapsto (3,1);3\mapsto (2,0); o\mapsto (o,\infty).$$ The map $f$ is the composition of a permutation map $g: 1\mapsto (1,0);2\mapsto (3,0);3\mapsto (2,0)$ and an identity map with noninfinite values in $h_2$,  $1\mapsto (1,3),2 \mapsto (2,1), 3\mapsto (3,0), o\mapsto (o,\infty)$, and $f=h\circ g$ as defined above.  Hence the valuations of sets in $f^{-1}(\mu)$ are $f^{-1}(\mu)(12)=0+3+1 = 4$, $f^{-1}(\mu)(13)=0+3+0 = 3$, $f^{-1}(\mu)(23)=0+0+1 = 1$, and $f^{-1}(\mu)(I) = \infty$ if $o\in I$. The following table gives the cocircuits of $f^{-1}(\mu)$ and the vectors $\val(A)\odot C^*_{\mu}(I)$.\begin{center}
    \begin{tabular}{c|c|c}
        $I$ & $C^*_{f^{-1}(\mu)}(I)$ & $\val(A)\odot C^*_{\mu}(I)$ \\\hline 
        $1$& $(\infty,4,3)$& $(\infty,1,0)$\\
        $2$& $(4,\infty,1)$ & $(3,1,\infty)$ \\
        $3$ & $(3,1,\infty)$ & $(3,\infty,0)$
    \end{tabular}
   \end{center} Note that $C^*_{f^{-1}(\mu)}(1) =\val(A)\odot C^*_{\mu}(1) \odot 3 =\val(A)\odot C^*_{\mu}(1) \odot f_2(1)$, and that analogously, $C^*_{f^{-1}(\mu)}(2) =\val(A)\odot C^*_{\mu}(3) \odot 1 =\val(A)\odot C^*_{\mu}(1) \odot f_2(2)$ and $C^*_{f^{-1}(\mu)}(2) =\val(A)\odot C^*_{\mu}(2) =\val(A)\odot C^*_{\mu}(1) \odot f_2(3)$, as in the proof of Lemma \ref{lemma: mat equals morph diags}.
\end{example}
 \section{Quiver Dressians}\label{sec quiver dressians}

 \subsection{Quiver Grassmannians and representations}\label{sec quiver grassmannian}
We first collect some facts about quiver representations. Standard references are \cite{crawley1992lectures,schiffler2014quiver}.

\begin{definition}
    A finite \defemph{quiver} $Q=(V,A,s,t)$ is a directed graph given by a finite set of vertices $V$, a finite set of arrows $A$ and two maps $s,t: A\to V$ assigning to each arrow its source, resp. target. 
 \end{definition}

 \begin{definition}
     Given a quiver $Q$, we define a finite-dimensional $Q$-\defemph{representation} $R$ over a field $K$ as the ordered pair $((R_i)_{i\in V},(R^{\alpha})_{\alpha\in A})$, where $R_i$ is a finite-dimensional $K$-vector space attached to vertex $i\in V$ and $R^{\alpha}:R_{s(\alpha)}\to R_{t(\alpha)}$ is a $K$-linear map for any $\alpha \in A$.
     
    The \defemph{dimension vector} of $R$ is $\dim(R) \coloneqq(\dim_{K}(R_i))_{i\in V} \in \ZZ^{V}_{\geq 0}$.
     A \defemph{subrepresentation} of $R$ is a $Q$-representation $N$ $=((N_i)_{i\in V},(R^{\alpha}|_{N_{s(\alpha)}})_{\alpha\in A})$
     such that $N_i\subseteq R_i \; $ for all $ i\in V$ and $R^{\alpha}(N_{s(\alpha)})\subseteq N_{t(\alpha)}\; $ for all $ \alpha\in A$. 
\end{definition}

For the purposes of this paper, from now on we will only consider quiver representations with $R_i=K^n$ for all $i \in V$. In this setting, subrepresentations of $R$ describe the relative containment of subspaces via  fixed linear maps inside a common ambient vector space.

 \begin{example}\label{ex:twotowersriver}
     Let us fix a basis $\mathcal{B}=\{ b_1,b_2,b_3,b_4 \}$ of $\CC^4$ and consider the following quiver with its representation $R=((\CC^4)_{i\in[4]},(\id)_{j\in[4]})$:

\begin{equation*}
R: \begin{tikzcd}[ampersand replacement=\&, column sep= huge]
 \& \overset{\CC^4}{\bullet} \ar[dr,"{\begin{bsmallmatrix}
1&0&0&0\\
0&1&0&0\\
0&0&1&0\\
0&0&0&1
\end{bsmallmatrix}}"] \& \\
\overset{\CC^4}{\bullet} \ar[ur,"{\begin{bsmallmatrix}
1&0&0&0\\
0&1&0&0\\
0&0&1&0\\
0&0&0&1
\end{bsmallmatrix}}"] \ar[dr, "{\begin{bsmallmatrix}
1&0&0&0\\
0&1&0&0\\
0&0&1&0\\
0&0&0&1
\end{bsmallmatrix}}"']\& \& \overset{\CC^4}{\bullet}\\
\& \overset{\CC^4}{\bullet} \ar[ur,"{\begin{bsmallmatrix}
1&0&0&0\\
0&1&0&0\\
0&0&1&0\\
0&0&0&1
\end{bsmallmatrix}}"']
\end{tikzcd}.
\end{equation*}
An example of a subrepresentation $N$ of $R$ is $N=((\langle b_1\rangle,\langle b_1, b_2\rangle,\langle b_1,b_4\rangle,\langle b_1,b_2,b_4\rangle),\\(\begin{bsmallmatrix}
1\\
0
\end{bsmallmatrix},\begin{bsmallmatrix}
1\\
0
\end{bsmallmatrix},\begin{bsmallmatrix}
1&0\\
0&1\\
0&0
\end{bsmallmatrix},\begin{bsmallmatrix}
1&0\\
0&0\\
0&1
\end{bsmallmatrix}))$,
which has dimension vector $\dim (N)=(1,2,2,3)$. The matrices representing the linear maps appearing in $N$ are the restrictions of the identity maps to the chosen subspaces.

A sequence of subspaces that cannot belong to any subrepresentation of $R$ is, for instance, $(\langle b_1\rangle,\langle b_1, b_2\rangle,\langle b_1,b_4\rangle,\langle b_1,b_2,b_3\rangle)$, because $\id(\langle b_1,b_4\rangle)\nsubseteq \langle b_1,b_2,b_3\rangle$.
 \end{example}

\begin{definition}[Quiver Grassmannians]
    Consider a quiver $Q$, a $Q$-representation $R$ and a dimension vector $\mathbf{d}\in\ZZ^{V}_{\geq 0}$ such that $d_i\leq n \; $ for all $ i\in V$. The \defemph{quiver Grassmannian} in dimension $n$, denoted by $\QGr(R,\mathbf{d};n)$, is defined as the collection of all subrepresentations $N$ of $R$ with $\dim N_i=d_i \; $ for all $ i\in V$.
\end{definition}

From now on, we will consider quiver Grassmannians as the zero locus of quiver Pl\"ucker relations (see Definition \ref{def: quiver pluecker relations}), since they coincide pointwise and their scheme-theoretic structure as (possibly not reduced) projective varieties is not relevant to the purposes of this paper.

\begin{example}[The flag variety]\label{ex: flag variety}

    In \cite[Proposition 2.7]{cerulliirelli2012quiveranddegenerate}, the authors realise the (linear degenerate) flag variety as the quiver Grassmannian associated to representations of the equioriented quiver of type $A_n$. In particular, the complete flag variety can be realised as follows.

    Consider the quiver with $n$ vertices, ordered from 1 to $n$, and $n-1$ arrows of the form $i\to i+1$. We fix the  dimension vector $d=(1,2,\dots,n)$ and the representation $R$ with $R_i=\CC^{n+1}$ for $i=1,\dots,n$ and $R^{\alpha}=\id$ for all $\alpha\in A$:
\begin{equation*}
\begin{tikzcd}[]
\overset{\CC^{n+1}}{\bullet} \ar[r, "\id"] & \overset{\CC^{n+1}}{\bullet} \ar[r, "\id"] & ... \ar[r, "\id"] & \overset{\CC^{n+1}}{\bullet}.\\
\end{tikzcd}
\vspace{-0.7cm}
\end{equation*}
 The quiver Grassmannian $\QGr(R,\mathbf{d};n+1)$ consists precisely of the subrepresentations $N$ of $R$ with $\dim(N_i)=i$ and $N_i\subseteq N_{i+1}$, i.e. flags of vector subspaces.
\end{example}

    Analogous to Grassmannians and flag varieties, quiver Grassmannians can be realized pointwise as subvarieties of products of projective spaces, via the closed embedding
    \begin{equation*}
        \iota: \QGr(R,\mathbf{d};n) \to \prod_{i \in V}\Gr(d_i,K) \subseteq \PP_K^{\binom{n}{d_1}}\times \PP_K^{\binom{n}{d_2}}\times \dots \times \PP_K^{\binom{n}{d_{|V|}}}
    \end{equation*}
    which sends a subrepresentation $N$ of $R$ to the collection of $d_i$-dimensional subspaces $N_i$ of $R_i$.
    Relations that define the subvariety associated to a quiver representation and a dimension vector can be given pointwise as follows.
    
    	\begin{definition}[Quiver Pl\"{u}cker relations, \cite{lorscheid2019pluckerelations}]\label{def: quiver pluecker relations}
		Let $Q=(V,A,s,t)$ be a quiver and $R$ a $Q$-representation. After fixing bases for all vertices, let $M_{\alpha}$ be the matrix of the map of $\alpha\in A$. Let $r = \dim(s(\alpha))$ and $s = \dim(t(\alpha))$. %Let $\mathscr{B}_v \subseteq[n]$ denote the set of indices of nonzero columns of $v\in V$. 
    For each arrow $\alpha$, the \defemph{quiver Pl\"{u}cker relations} are the polynomials in the variables $\{ p_I : I \in \binom{[n]}{r} \} \cup \{ p_J : J \in \binom{[n]}{s} \}$ with coefficients in $K$: 
		\[ \mathscr{P}_{\alpha;n} = \left\{ \sum_{j \in [n]\setminus I,  i \in J} \sign(j;I,J) (M_{\alpha})_{i,j} p_{I \cup j} p_{J \setminus i} : I \in \binom{[n]}{r-1}, J \in \binom{[n]}{s+1}  \right\} \]
		where $\sign(j;I,J) = (-1)^{ \#\{ j' \in J : j < j' \} + \# \{ i \in I : i > j \}  }$. Their tropicalization will be denoted by $\mathscr{P}_{\alpha;n}^{\trop}$. The Pl\"ucker relations corresponding to the vertices are the standard Grassmann-Pl\"ucker relations of the associated vector spaces. In \cite{lorscheid2019pluckerelations}, it is shown that the vanishing set of the quiver Pl\"ucker relations coincides with the associated quiver Grassmannian.
	\end{definition}

\subsection{Tropicalized quiver Grassmannians} \label{sec tropicalized quiver grassmannian}

Now, we come to the main part of the paper: we combine  quiver Grassmannians and tropical geometry. We substitute the linear spaces we had assigned to the vertices by tropical linear spaces, and the matrix multiplication by the \textit{tropical matrix multiplication} with valuated matrices we discussed in Section \ref{sec matrixmult}. 

 \begin{definition}\label{def tropicalized quiver grassmannian}
     Let $Q$ be a quiver and let $M$ be a $Q$-representation with quiver Grassmannian $\QGr(M,\mathbf{d};n)$. The \emph{tropicalized quiver Grassmannian} $$\tropbar\big(\QGr(M,\mathbf{d};n)\big)\subseteq \PP\big(\TT^{\binom{n}{d_1}}\big)\times \cdots \times \PP\big(\TT^{\binom{n}{d_{|V|}}}\big)$$ is the tropicalization of $\QGr(M,\mathbf{d};n)$.
 \end{definition}

 We now show that the tropicalized quiver Grassmannian parametrizes containment of tropicalized linear spaces under tropical matrix multiplication, i.e. we prove the equivalence $(a)\Leftrightarrow(b)$ of Theorem \ref{thm: main_quiver_correspondence_realizable}.
 \begin{proposition}
      Let $K$ be an algebraically closed field with nontrivial valuation, and let $M$ be a quiver representation of a quiver $Q$ with quiver Grassmannian $\QGr(M,\mathbf{d};n)$, for some dimension vector $d$. Then, $p\in \tropbar\big(\QGr(M,\mathbf{d};n)\big)$ if and only if there exists a tropical linear space $\tropbar(\mu_i)$ for each vertex $i\in V$ such that $\val(M_f)\odot\tropbar(\mu_{s(f)})\subseteq \tropbar(\mu_{t(f)})$ for each arrow $f$, and there exists a quiver subrepresentation $N=\big((N_i)_{i\in V},(M^{\alpha}|_{N_{s(\alpha)}})_{\alpha\in A}\big)$ over $K$ such that $
      \tropbar(\mu_i) = \tropbar(N_i)$ for all $i\in V$.
 \end{proposition}
\begin{proof}
        For ease of notation, we restrict to the case where $Q$ is a graph with two vertices and one arrow $f$, and we write $\QGr(M_f,\mathbf{d};n)$ for the corresponding quiver Grassmannian. All other cases follow similarly. 
		If $\mu\times \nu \in \tropbar(\QGr(M_f,\mathbf{d};n))$, from the Fundamental Theorem of Tropical Geometry \cite[Theorem 6.2.15]{maclagan2015book}, there exist realizations $U$ of $\mu$ and $V$ of $\nu$ such that the Pl\"{u}cker coordinates of $U$ and $V$ are a point of $\QGr(M_f,\mathbf{d};n)$. By the main theorem in \cite{lorscheid2019pluckerelations}, points in $\QGr(M_f,\mathbf{d};n)$ satisfy $M_f\cdot U \subseteq V$, thus $\val(M_f)\odot \tropbar(U)\subseteq\tropbar(V)$.
		
		Now conversely assume that  $\val(M_f)\odot \tropbar(U)\subseteq\tropbar(V)$, and that there exist realizations $U$ and $V$ such that $M_f\cdot U \subseteq V$. Then, $U\times V \in \QGr(M_f, \mathbf{d};n)$, hence $p_{\tropbar(U)}\times p_{\tropbar(V)}\in \tropbar(\QGr(M_f,\mathbf{d};n)))$ by the Fundamental Theorem.
\end{proof}

\subsection{Quiver Dressians} \label{subsec quiver dressians}
Instead of parametrizing tropicalized linear spaces and their containment relations, we now consider parameter spaces of their tropical analogues and show Theorems \ref{thm: main_quiver_correspondence} and \ref{thm: main_monomial_matrices}.

 \begin{definition}\label{def: quiver dressian}
     Let $Q = (V,A,s,t)$ be a quiver. The \emph{quiver Dressian} 
     $\QDr(R,\mathbf{d};n) \subseteq \mathbb{T}^{d_1}\times \cdots \times \mathbb{T}^{d_m}$ 
     is the tropical prevariety cut out by the tropical Pl\"ucker relations and the tropical quiver Pl\"ucker relations, $\{\mathscr{P}_{d_i;n}^{\trop}\}_{i \in V} \cup \{\mathscr{P}_{\alpha;n}^{\trop}\}_{\alpha \in A}$ (see Definition \ref{def: quiver pluecker relations}).
 \end{definition}

Now we show  Theorem \ref{thm: main_quiver_correspondence}, i.e. that the quiver Dressian parametrizes containment of tropical linear spaces under matrix multiplication.

\begin{theorem}\label{prop: equivalence quiver dressian containment after tropical flag}
Let $\mu$ and $\nu$ be valuated matroids over the ground set $[n]$ of rank $r$ and $s$ respectively, and let $Q$ be a quiver consisting of two vertices connected by one arrow $f$. Let $M$ denote a $Q$-representation assigning the matrix $M_f$ to $f$. Then, 
$$\mu\times\nu \in \QDr(M,(r,s);n) \Leftrightarrow \val(M_f)\odot\tropbar(\mu)\subseteq\tropbar(\nu).$$\end{theorem}
\proof The standard Grassmann-Pl\"ucker relations associated to the vertices vanish if and only if  $\mu$ and $\nu$ are valuated matroids. Thus, we only focus on the quiver Pl\"ucker relations. 
By definition, $\mu\times\nu \in \QDr(M,(r,s);n)$ if and only if for all $I\in\binom{[n]}{r-1}$ and $J\in\binom{[n]}{s+1}$, the minimum in
\begin{align*}
 \bigoplus_{j\in [n]\setminus I, i \in J } \big( \val((M_{f})_{i,j})\odot p_{I\cup j}\odot p_{J\setminus i} \big)
 \end{align*}
   is attained at least twice. Equivalently, for all $I$ and $J$ as above, the minimum in 
 \begin{align}\label{eq vaiety}
\bigoplus_{\substack{j\in[n]\setminus I,\\ i \in J}} \big( \val((M_{f})_{i,j})\odot \mu(I\cup j)\odot \nu(J\setminus i) \big) =  \bigoplus_{\substack{j\in[n]\setminus I,\\ i \in J}} \big( \val((M_{f})_{i,j})\odot C^{*}_{\mu}(I)_j\odot C_{\nu}(J)_i \big)
\end{align}
 is attained at least twice. We write $\val(M_f)\odot C^{*}_{\mu}(I)$ for the vector with coordinate entries $$\big(\val(M_{f})\odot C^{*}_{\mu}(I)\big)_j := \big(\val(m_{f,1,j})\odot C^{*}_{\mu}(I)_j\big)\oplus \cdots \oplus \big(\val(m_{f,n,j})\odot C^{*}_{\mu}(I)_j\big).$$  By distribution, the minimum in \eqref{eq vaiety} is attained twice if and only if 
$$\val(M_{f})\odot C^{*}_{\mu}(I)\in V\big(\bigoplus_{i \in [n]} C_{\nu}(J)_i\odot x_i\big)=\tropbar(\nu).$$
Finally,  by Lemma \ref{lem: image is tropically convex}, $\val(M_f)\odot\tropbar(\mu)$ is tropically convex. Using Proposition-Definition \ref{thm: characterizations of trop mu}, the above is thus equivalent to $$\Big\{\bigoplus_{C^*_{\mu}(I)\in \C^{*}(\mu)}\lambda_{C^*_{\mu}(I)}\odot\val(M_{f})\odot C^{*}_{\mu}(I)\hspace{3pt}:\hspace{3pt}\lambda_{C^*_{\mu}(I)}\in\RR\Big\}=\val(M_f)\odot\tropbar(\mu)\subseteq\tropbar(\nu).$$
\endproof

\begin{corollary}\label{theorem b}
    Let $M_f$ be a square weakly monomial matrix, and let $\mu$ and $\nu$ be matroids of ranks $r$ and $s$ over $[n]$. Then, $\val(M_f)\odot\tropbar(\mu)\subseteq\tropbar(\nu)$ if and only if $f:\nu\rightarrow\mu\times\RR$ as constructed in Proposition \ref{prop affine morphism is matrix multiplication} is an affine morphism of valuated matroids. Further, $\val(M_f)\odot\tropbar(\mu)\subseteq\tropbar(\nu)$ is realizable if and only if $f:\nu\rightarrow\mu\times\RR$ is a realizable affine morphism of valuated matroids.
\end{corollary}
\begin{proof}
    By Proposition \ref{prop affine morphism is matrix multiplication}, there exists a map $f:[n]\cup \{o\}\rightarrow[n]\cup\{o\}\times\mathbb{T}$ such that $\val(M_f)\odot\tropbar(\mu) = \tropbar\big(f^{-1}(\mu)\big)$. 
    By \cite[Theorem A]{brandt2021tropicalflag},  $\tropbar\big(f^{-1}(\mu)\big) =\val(M_f)\odot\tropbar(\mu)\subseteq\tropbar(\nu)$ implies that $f^{-1}(\mu)|_{[n]} \twoheadleftarrow \nu$, i.e. by Definition \ref{def: affinemorphism}, that $f^{-1}$ is an affine morphism of valuated matroids. The realizability statement follows from Definition \ref{def: affinemorphism}.
\end{proof}
This shows Theorem \ref{thm: main_monomial_matrices} and the equivalence $(b)\Leftrightarrow(c)$ in Theorem \ref{thm: main_quiver_correspondence_realizable}.

\section{Realizability in quiver Dressians}\label{sec realizability}

In Section \ref{sec: prelims tropicalprojectivespace}, we remarked the difference between intrinsically tropical and tropicalized objects, and distinguished the Dressian, parametrizing \emph{tropical} linear spaces, from the tropicalized Grassmannians, parametrizing \emph{tropicalized} linear spaces. We observe a similar distinction for quiver Dressians and tropicalized quiver Grassmannians. 

The first example of a nonrealizeable tropical linear space, i.e. a tropical linear space that is not the tropicalization of any linear space, occurs in ambient dimension $8$. The first nonrealizable flag of tropical linear spaces already occurs for ambient dimension $6$ (see \cite[Example 5.2.4]{brandt2021tropicalflag}). For arbitrary quivers, the ambient dimension of the first nonrealizable quiver subrepresentation is even smaller.

\begin{remark}
    For ambient dimension 1, there are no classical Pl\"ucker relations, and the only quiver Pl\"ucker relations are the monomial relations corresponding to the coordinates of the only point in the quiver Grassmannian. On the algebraic side, $\QGr(R,\mathbf{d},1)$ is a point, hence $\tropbar(\QGr(R,\mathbf{d},1))\subseteq \PP(\TT^1)\times \cdots \times \PP(\TT^1)$ is a point. Since $\PP(\TT^1) \times \cdots \times \PP(\TT^1)$ is also just a point, the containment is an equality, and therefore $\tropbar(\QGr(R,\mathbf{d},1)) = \QDr(R,\mathbf{d},1)$.
\end{remark}

\begin{manualtheorem}{D}
     Let $Q$ be a finite quiver. For $n\geq2$, $\tropbar(\QGr(R, \mathbf{d};n))\subseteq\QDr(R,\mathbf{d};n)$ and there exist quiver representations $R$ where the containment is strict. For $n=1$, $\tropbar(\QGr(R, \mathbf{d};1))=\QDr(R,\mathbf{d};1)$ for any quiver $Q$ and any $Q$-representation $R$.
\end{manualtheorem}

The example for $\tropbar(\QGr(R, \mathbf{d};2))\neq\QDr(R,\mathbf{d};2)$ in ambient dimension 2 relies on the nontrivial valuation of the base field. Similar constructions can be given for higher ambient dimension, as described in \ref{ex nonrealizable n = 3}. However, the examples we construct afterwards for higher ambient dimension ($n\geq 4$) already occur for fields with trivial valuation, and their quivers have no parallel edges.  

\begin{example}\label{counterexample nonrealizable n 2}
We construct an example for $\tropbar(\QGr(R, \mathbf{d};2))\neq\QDr(R,\mathbf{d};2)$. The quiver we consider is known as the Kronecker quiver; we define its representation $R$ as shown in Figure \ref{fig: nonrealizable n 2}, with quiver Grassmannian $\QGr(R, (1,1);2)$ (see for instance \cite[Example 5]{irelli2020lectures}; in this case, we replace $\mathbb{C}$ with $\mathbb{C}\{\hspace{-3pt}\{ t \}\hspace{-3pt}\}$, the field of Puiseux series). It is an example of a reduced quiver Grassmannian of dimension 0 with two connected components (the two eigenspaces of the map corresponding to the lower arrow).

    Let $v_1$ and $v_2$ denote the Pl\"ucker variables of the space corresponding to the left vertex, and let $w_1$ and $w_2$ denote the Pl\"ucker variables of the right vertex.  Since $G(1;2)$ and $G(2;2)$ have no Grassmann-Pl\"ucker relations, the only relations are the quiver Pl\"ucker relations (see Definition \ref{def: quiver pluecker relations}), which are 
        $v_1w_2 + v_2w_1 \text{ and } v_1w_2 + (1+t)v_2w_1$.
    We have $$V(\langle v_1w_2 + v_2w_1, v_1w_2 + (1+t)v_2w_1 \rangle ) = \{((1:0),(1:0)), ((0:1),(0:1))\subset \PP^1 \times \PP^1\}, \text{ and }$$
    $$\tropbar(\QGr(R,(1,1),2)) = \{((0:\infty),(0:\infty)), ((\infty:0),(\infty:0))\subset \PP(\TT^2) \times \PP(\TT^2)\}.$$

    Tropicalizing the generators, we have that $V(\tropbar(v_1 w_2+v_2 w_1))$ is the set  $$W = \{((v_1:v_2),(w_1:w_2))\in \PP(\TT^2)\times\PP(\TT^2)|\min(v_1+w_2, v_2+w_1) \text{ is attained at least twice}\}.$$ Since $\val(1+t) = \val(1) = 0$, we further have $W = V(\tropbar(v_1w_2 + (1+t)v_2w_1))$, thus $W=\QDr(R,(1,1),2)$. Now, $W$ can be rewritten as $$W=\{((v_1:v_2),(w_1:w_2))\in  \PP(\TT^2)\times\PP(\TT^2) \;  | \; v_1+w_2 = v_2+w_1\},$$ which is a connected 1-dimensional space (that contains the two points above), whereas $\tropbar(\QGr(R,(1,1),2))$ is not.
\end{example}

\begin{figure}
    \centering
    \begin{tikzpicture}
        \draw[fill] (0,0) circle (2pt); 
        \draw[fill] (4,0) circle (2pt);
        \draw[->] (0.2,0.2) -- (3.8, 0.2);
        \draw[->] (0.2,-0.2) -- (3.8,-0.2);
        \node[above] at (2,0.3){$\begin{bmatrix}
        1 & 0 \\ 0&1
    \end{bmatrix}$};  \node[below] at (2,-0.3){$\begin{bmatrix}
        1 & 0 \\ 0&1+t
    \end{bmatrix}$};
    \node[above] at (0,0.3){$\mathbb{C}\{\hspace{-3pt}\{ t \}\hspace{-3pt}\}^2$};
    \node[above] at (4,0.3){$\mathbb{C}\{\hspace{-3pt}\{ t \}\hspace{-3pt}\}^2$}; \node[below] at (0,-0.3){$[1]$};
    \node[below] at (4,-0.3){$[1]$};
    \end{tikzpicture}
    \caption{A quiver $Q$ with $Q$-representation $R$ for $n=2$ where $\tropbar(\QGr(R, (1,1); 2)) \neq \QDr(R,(1,1); 2)$.}
    \label{fig: nonrealizable n 2}
\end{figure}

\begin{example}\label{ex nonrealizable n = 3}
To obtain an analogous example for ambient dimension 3, we can consider the same quiver as in Example \ref{counterexample nonrealizable n 2}. We assign $\mathbb{C}\{\hspace{-3pt}\{ t \}\hspace{-3pt}\}^3$ to each vertex, the matrix $\begin{bsmallmatrix}
        1 & 0 & 0 \\ 0&1 & 0 \\ 0 & 0 &1
\end{bsmallmatrix}$ to the upper arrow and $\begin{bsmallmatrix}
        1 & 0 & 0 \\ 0&1+t & 0 \\ 0 & 0 &1+t^2
\end{bsmallmatrix}$ to the lower arrow. Again, $G(1;3)$ has no Grassmann-Pl\"ucker relations, so the only Pl\"ucker relations are $$v_1w_2 - v_2w_1, v_1w_3+v_3w_1, v_2w_3-v_3w_2$$ for the upper arrow, and $$v_1w_2 - (1+t)v_2w_1, v_1w_3+ (1+t^2)v_3w_1, (1+t)v_2w_3-(1+t^2)v_3w_2$$ for the lower arrow. The zero locus of the six equations is zero-dimensional and consists of three points: $((1:0:0),(1:0:0)),((0:1:0),(0:1:0))$ and $((0:0:1),(0:0:1))$, so the tropicalization of the quiver Grassmannian does too. Again, as the valuations of all nonzero matrix entries is zero, the quiver Dressian is the set $$ \{(\mathbf{v},\mathbf{w})\in  \PP(\TT^3)\times\PP(\TT^3) \; | \; v_1+w_2 = v_2+w_1, v_1+w_3 = v_3+w_1 \text{ and } v_2 + w_3 = w_2 + v_3\}.$$
This set is 1-dimensional, thus the tropicalized quiver Grassmannian and the quiver Dressian differ. This example can similarly be extended to an example for higher ambient dimension $n$. Here, we assign $\mathbb{C}\{\hspace{-3pt}\{ t \}\hspace{-3pt}\}^n$ to both vertices, consider the dimension vector $(1,1)$ and assign the matrices to the two arrows as follows: one arrow is assigned the identity matrix, and the other arrow gets the diagonal matrix with entries $(1,1+t,1+t^2,\dots, 1+t^{n-1})$.
An example for a quiver representation of trivially valued fields over a quiver with no parallel edges can also be found, though it is significantly more complicated. We give an example at \url{https://victoriaschleis.github.io/code}. 
\end{example}

Now we give an example of a quiver representation $R$ satisfying $\tropbar(\QGr(R, \mathbf{d};4))\neq\QDr(R,\mathbf{d};4)$ over a field with trivial valuation, using a quiver without parallel edges. Afterwards, we will extend this to a family of such examples for $n>4$. 
\begin{example}\label{ex:nonrealizable quiver n = 4}
We return to the quiver given in Example \ref{ex:twotowersriver}. 

\begin{equation*}
Q,M: \begin{tikzcd}[ampersand replacement=\&, column sep= large, row sep = small]
 \& \overset{\CC^4}{\bullet} \ar[dr,"M_{\id}"] \& \\
\overset{\CC^4}{\bullet} \ar[ur,"M_{\id}"] \ar[dr, "M_{\id}"]\& \& \overset{\CC^4}{\bullet}\\
\& \overset{\CC^4}{\bullet} \ar[ur,"M_{\id}"]
\end{tikzcd}, \;
M_{\id}=\begin{bsmallmatrix}
1&0&0&0\\
0&1&0&0\\
0&0&1&0\\
0&0&0&1\\
\end{bsmallmatrix},\; d=(1,2,2,3).
\end{equation*}

This quiver Grassmannian parametrizes the arrangement of four tropical objects: two tropical lines that are contained in a common tropical plane, and a common point lying on all of them. In Figure \ref{fig:enter-label}, we give an example of such an arrangement.
 
From Definition \ref{def: quiver pluecker relations}, we obtain the following equations for the quiver Grassmannian $\QGr(M,d;4)$ inside the product of Grassmannians $\prod \Gr(d_i;4)$:
\begin{center}
\begin{tabular}{c c c}
    $p_{12}p_{34}-p_{13}p_{24}+p_{14}p_{23}$ & $p'_{12}p'_{34}-p'_{13}p'_{24}+p'_{14}p'_{23}$& $p_{12}p_{134}+p_{13}p_{124}+p_{14}p_{123}$ \\ $p_{12}p_{234}+p_{23}p_{124}+p_{24}p_{123} $ &$ p_{13}p_{234}+p_{23}p_{134}+p_{34}p_{123} $&$ p_{14}p_{234}+p_{24}p_{134}+p_{34}p_{124} $\\$ p'_{12}p_{134}+p'_{13}p_{124}+p'_{14}p_{123} $&$p'_{12}p_{234}+p'_{23}p_{124}+p'_{24}p_{123}$&$p'_{13}p_{234}+p'_{23}p_{134}+p'_{34}p_{123}$\\$p'_{14}p_{234}+p'_{24}p_{134}+p'_{34}p_{124}$&$p_{1}p_{23}+p_{2}p_{13}+p_{3}p_{12}$&$p_{1}p_{24}+p_{2}p_{14}+p_{4}p_{12}$\\$p_{1}p_{34}+p_{3}p_{14}+p_{4}p_{13}$&$p_{2}p_{34}+p_{3}p_{24}+p_{4}p_{23}$&$p_{1}p'_{23}+p_{2}p'_{13}+p_{3}p'_{12}$\\$p_{1}p'_{24}+p_{2}p'_{14}+p_{4}p'_{12}$&$p_{1}p'_{34}+p_{3}p'_{14}+p_{4}p'_{13}$&$p_{2}p'_{34}+p_{3}p'_{24}+p_{4}p'_{23}$
\end{tabular}
\end{center}
where we denote by $p'_{ij}$ the Pl\"{u}cker coordinates corresponding to the two-dimensional  subspace in the bottom row, denoted by $\Gr(d_3;4)$.

We use the code provided at \url{https://victoriaschleis.github.io/code} to compute the quiver Dressian and the tropicalized quiver Grassmannian in \texttt{gfan} \cite{gfan}, and do some auxiliary computations in \textsc{Oscar} \cite{Oscar}.

The quiver Dressian has dimension $12$ and f-Vector $(1,58,466,1156,858,3)$. The tropicalized quiver Grassmannian has dimension $10$, as does the ideal generated by the polynomials. Since the dimensions of the tropical (pre-)varieties differ, they cannot be equal, showing the second part of Theorem \ref{thm: main_realizability} for $n=4$. 

As a polyhedral complex, the tropicalized quiver Grassmannian is the union of the tropicalization of the $46$ primary components of the quiver Grassmannian. Of these components, 37 tropicalize to linear components of dimensions $8$, $7$, $6$ and $5$ in different coordinate directions. Each of the remaining nine components has, after quotienting out lineality, six rays and ten  facets, whose incidences are depicted in the graph in Figure \ref{fig:nonrealizable quiver n = 4 strata}.

\end{example}
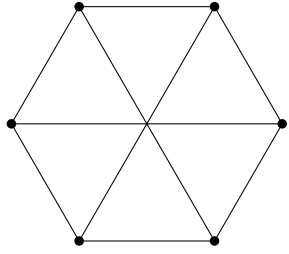
\begin{figure}
    \centering
    \begin{tikzpicture}[scale = 0.9]
    \foreach \i in {1,...,6}
         \fill (\i*360/6:2) coordinate (6\i) circle(2 pt);
         
    \draw (360/6:2) -- (2*360/6:2) -- (3*360/6:2) -- (4*360/6:2) -- (5*360/6:2) -- (0:2) --cycle
    (360/6:2) -- (4*360/6:2) 
    (2*360/6:2) -- (5*360/6:2)(3*360/6:2) --(0:2) ;
    \end{tikzpicture}
    \caption{The nonlinear irreducible components of Example \ref{ex:nonrealizable quiver n = 4} are linear spaces of dimension $10$ and $8$ over the graph above.} 
    \label{fig:nonrealizable quiver n = 4 strata}
\end{figure}

\begin{corollary}
 For $n\geq4$, there exists a quiver $Q$ with a representation $M'$ such that $\tropbar(\QGr(M',\mathbf{d};n))\subsetneq\QDr(M',\mathbf{d};n)$.   
\end{corollary}
\begin{proof}
Let $n\geq 4$. We consider the quiver representation $M$ of Example \ref{ex:nonrealizable quiver n = 4}, and construct a quiver representation $M'$ by substituting each base set on the vertices by $[n]$. For each matrix, we append $n-4$ zero columns. This way, $\QGr(M',\mathbf{d};n)$ has the same Pl\"ucker relations as $\QGr(M,\mathbf{d};4)$. Since $\dim(\tropbar(\QGr(M,\mathbf{d};4)) < \dim(\QDr(M,\mathbf{d};4))$, we obtain  that $\dim(\tropbar(\QGr(M',\mathbf{d};n))) < \dim(\QDr(M',\mathbf{d};n))$. Hence, $\tropbar(\QGr(M',\mathbf{d};n))\subsetneq\QDr(M',\mathbf{d};n)$.
\end{proof}

This concludes the proof of Theorem \ref{thm: main_realizability}.

	\bibliographystyle{plain}
	\bibliography{bibliography.bib}

	\vspace{0.25cm}
	\noindent
	\textsc{Giulia Iezzi, RWTH Aachen, Pontdriesch 10-16, 52062 Aachen, Germany}\\
	\textit{Email:} iezzi@art.rwth-aachen.de
	\vspace{0.25cm}
	
	\noindent
	\textsc{Victoria Schleis, Universit\"at T\"ubingen, Auf der Morgenstelle 10,
		72076 T\"ubingen,
		Germany} \\
	\textit{Email :} victoria.schleis@student.uni-tuebingen.de
\end{document}